\documentclass[a4paper,twoside,10pt,leqno,
]{amsart}
\usepackage{amsmath,amsthm,amssymb,enumerate,url,float, graphicx}
\usepackage[mathscr]{eucal}

\swapnumbers
\theoremstyle{plain}
\newtheorem{theorem}{Theorem}[section]

\newtheorem{corollary}[theorem]{Corollary}

\theoremstyle{definition}

\newtheorem{definitions}[theorem]{Definitions}
\newtheorem{example}[theorem]{Example}

\newtheorem{examples}[theorem]{Examples}
\newtheorem{notation}[theorem]{Notation}

\newtheorem{remark}[theorem]{Remark}

\newtheorem{consequence}[theorem]{Consequence}

\newcommand{\gp}[2]{\gen{{#1}\mid #2}}
\newcommand{\gen}[1]{\langle\mkern3mu#1\mkern3mu\rangle}

\def \bbl1 {[\mkern-3mu[}
\def \bbr1 {]\mkern-3mu]}

\def \reals {\mathbb{R}}

\def \integers {\mathbb{Z}}

\def\d1{\discretionary{-}{}{-}}
\def\coloneq{\mathrel{\mathop\mathchar"303A}\mkern-1.2mu=}

\DeclareMathOperator{\Out}{Out}
\DeclareMathOperator{\Aut}{Aut}

\DeclareMathOperator{\GL}{GL}
\DeclareMathOperator{\Stab}{Stab}
\def\M{}
\DeclareMathOperator{\N}{Nielsen}
\DeclareMathOperator{\iN}{\mathit{Nielsen}}

\renewcommand{\phi}{\varphi}
\renewcommand{\epsilon}{\varepsilon}
\renewcommand{\le}{\leqslant}
\renewcommand{\ge}{\geqslant}

\begin{document}

\title[The Zieschang-M{\tiny c}Cool method]
{The Zieschang-M{\tiny c}Cool method for \\generating algebraic mapping-class groups}

\author{Llu\'{\i}s Bacardit and  Warren Dicks}

\date{\small\today}

\begin{abstract} Let $g, p \in [0{\uparrow}\infty\,[$, the set of non-negative integers.
Let $\operatorname{\normalfont A}_{g,p}$ denote
the group consisting of all those automorphisms of the free group on\linebreak
$ t_{[1 \uparrow p ]}  \cup x_{[1 \uparrow g ]} \cup  y_{[1 \uparrow g ]}$
which fix the element
\mbox{$\textstyle  \operatornamewithlimits{\Pi}\limits_{j \in [p\downarrow 1]} t_j \hskip-2pt
\operatornamewithlimits{\Pi}\limits_{i \in [1 \uparrow g]} [x_i,y_i] $}
and permute  the set
of  conjugacy classes \mbox{$\{\, [t_j] : j \in [1{\uparrow}p]\}$}.

Labru\`ere and Paris,  building on work of
\textbf{A}rtin,  Magnus, \textbf{D}ehn,  Nielsen,  \textbf{L}ickorish, Zieschang,  Birman,
\textbf{H}umphries, and others,
showed that $\operatorname{\normalfont A}_{g,p}$ is generated by what is called the   ADLH set.
We use  methods of   Zieschang and McCool  to
give a self\d1contained,  algebraic proof of this
result.

 Labru\`ere and Paris also gave  defining relations for  the ADLH set in $\operatorname{\normalfont A}_{g,p}$;
we do not know an algebraic proof of this for $g \ge 2$.

Consider an orientable surface $\mathbf{S}_{g,p}$
of genus $g$ with $p$ punctures,  with $(g,p) \ne (0,0)$, $(0,1)$.  The
\textit{algebraic}
 mapping-class group of $\mathbf{S}_{g,p}$, denoted  $\operatorname{\normalfont M}_{g,p}^{\text{alg}}$,  is defined as
the group of all those outer automorphisms of
$$ \gp{t_{[1 \uparrow p ]}  \cup x_{[1 \uparrow g ]} \cup  y_{[1 \uparrow g ]}}
{\textstyle \hskip-3pt \operatornamewithlimits{\Pi}\limits_{j \in [p\downarrow 1]} t_j \hskip-2pt
\operatornamewithlimits{\Pi}\limits_{i \in [1 \uparrow g]} [x_i,y_i]} \vspace{-1mm}$$
which permute  the set of conjugacy classes \mbox{$\{\, [t_j], [\overline t_j] : j \in [1{\uparrow}p]\}$}.
It now follows from a result of Nielsen that $\operatorname{\normalfont M}_{g,p}^{\text{alg}}$ is
generated by the image of the ADLH set together with a  reflection.
This gives a  new way of seeing
that    $\operatorname{\normalfont M}_{g,p}^{\text{alg}}$ equals  the  (topological)   mapping-class group  of
$\mathbf{S}_{g,p}$, along  lines
 suggested by  Magnus, Karrass, and Solitar  in 1966.
\medskip

{\footnotesize
\noindent \emph{2010 Mathematics Subject Classification.} Primary:  20E05;
Secondary: 20E36, 20F05, 57M60, 57M05.

\noindent \emph{Key words.}  Algebraic mapping-class group.  Zieschang groupoid.  Generating set.}

\end{abstract}

\maketitle

\section{Introduction}

Notation will be explained more fully in Section~\ref{sec:Z1}.

\begin{definitions}\label{defs:A}
Let  $g$, $p \in [0{\uparrow}\infty[$\,.
 Let $\operatorname{\normalfont A}_{g,p}$ denote the group of automorphisms of $ \gp{t_{[1 \uparrow  p ]}
 \cup x_{[1 \uparrow g ]} \cup  y_{[1 \uparrow g ]}}
{\quad} $
that fix  $\operatornamewithlimits {\Pi}\limits_{j \in [p\downarrow 1]} t_j \hskip-2pt
\operatornamewithlimits{\Pi}\limits_{i \in [1 \uparrow g]} [x_i,y_i] $ \vspace{-2mm} and
permute the set of  conjugacy classes \mbox{$\{\, [t_j] :  j \in [1{\uparrow}p]\}$}.

 We shall usually codify an element $\phi \in \operatorname{\normalfont A}_{g,p}$ as a two-row matrix where the
first row gives all the  elements  of $t_{[1 \uparrow  p ]}
 \cup x_{[1 \uparrow g ]} \cup  y_{[1 \uparrow g ]}$ that are moved by $\phi$, and the second row
equals the $\phi$-image of the first row.  We define the
 following elements of $\operatorname{\normalfont A}_{g,p}$: \vspace{-5pt}

\hskip  1cm for each \mbox{$j \in [2{\uparrow}p]$,}   \mbox{ $\sigma_j \coloneq \bigl(\begin{smallmatrix}
t_j  && t_{j-1} \\    t_{j-1}  &&  \overline t_{j-1}t_{j} {t_{j-1}}
\end{smallmatrix}\bigr);$}\vspace{.5mm}

\hskip  1cm  for each $i \in [1{\uparrow}g]$,
 \mbox{ $\alpha_i \coloneq  \bigl(\begin{smallmatrix}
x_i \\   \overline y_ix_i
\end{smallmatrix}\bigr)  $} and \mbox{$\beta_i \coloneq \left(\begin{smallmatrix}
y_i \\     x_iy_i
\end{smallmatrix}\right) $;}\vspace{.5mm}

\hskip 1cm for each $i \in [2{\uparrow}g]$,
\mbox{ $\gamma_i \coloneq \bigl(\begin{smallmatrix}
x_{i-1} & & y_{i-1} & & x_{i}\\   \overline w_i x_{i-1} & &  \overline w_i y_{i-1} {w_{i}} & & x_{i}w_{i}
\end{smallmatrix}\bigr)$} with \mbox{$w_i \coloneq y_{i-1}  \overline x_i \overline y_{i } {x_i}$;}

\hskip  1cm  if  $\min(1,g,p) = 1$,
 \mbox{ $\gamma_1 \coloneq \bigl(\begin{smallmatrix}
t_1 & & x_1  \\    \overline w_1 t_1 {w_1} & & x_{1}w_{1}
\end{smallmatrix}\bigr)$} with \mbox{$w_1 \coloneq  t_1 \overline x_1 \overline y_1 {x_1}$}.

\medskip

\noindent We say that
$
\sigma_{[2{\uparrow}p]} \cup \alpha_{[1{\uparrow}g]} \cup \beta_{[1{\uparrow}g]}
\cup \gamma_{[\max(2-p,1){\uparrow}g]}
$
 is the \textit{ADL  set}, and  that
 removing $\alpha_{[3{\uparrow}g]}$ leaves the  \textit{ADLH  set},
$ \sigma_{[2{\uparrow}p]} \cup \alpha_{[1{\uparrow}\min(2,g)]} \cup \beta_{[1{\uparrow}g]}
\cup  \gamma_{[\max(2-p,1){\uparrow}g]},$
named after
 Artin, Dehn, Lickorish and Humphries.
\hfill\qed
\end{definitions}

In~\cite[Proposition 2.10(ii) with $r=0$]{LP}, Labru\`ere and Paris  showed that
$\operatorname{\normalfont A}_{g,p}$ is generated by the ADLH set.  As we shall recall in Section~\ref{sec:mcg},
the   proof is built on work of
Artin,  Magnus, Dehn,  Nielsen,  Lickorish,  Zieschang,  Birman, Humphries, and others,
and some of this work uses topological arguments.

The main purpose of this article is to give a self\d1contained,  \textit{algebraic} proof
that $\operatorname{\normalfont A}_{g,p}$ is generated by the ADL set.
Such proofs were given
in the case $(g,p) = (1,0)$ by  Nielsen~\cite{Nielsen0},
and in the case $g=0$ by Artin~\cite{A},
 and in the case  $p=0$ by McCool~\cite{Mc}.
In the case where $(g,p) = (1,0)$ or $g=0$, our proof follows
Nielsen's and Artin's.
In the case where $p=0$, McCool  proceeds by adding in the free generators two at a time,
 while, for the  general case, we benefit  from being able to add in
 the free generators one at a time.

We also give a self-contained, algebraic translation of
Humphries' proof~\cite{Humphries} that the  ADLH  set then generates   $\operatorname{\normalfont A}_{g,p}$.

\begin{remark}   In~\cite[Theorem~3.1 with $r=0$]{LP},  Labru\`ere and Paris
use  topological and algebraic results of various authors to
  present
$\operatorname{\normalfont A}_{g,p}$  as the quotient of an Artin group on the ADLH   set
modulo three-or-less relations, each of which is expressed in terms of
centres of Artin subgroups.
We would find it very satisfying to have a direct, algebraic proof of this  beautiful presentation.
Now that we have the ADLH generating set,  it would suffice
to consider the group with the desired presentation and verify
 that its action  on $  \gp{t_{[1 \uparrow  p ]}
 \cup x_{[1 \uparrow g ]} \cup  y_{[1 \uparrow g ]}}
{\quad} $  is faithful.
This is precisely the approach carried out by Magnus~\cite{Magnus}  for both the case $g = 0$, see~\cite[Section~5]{BD},
and the case $g=1$, see~\cite[Section~6.3]{Bacardit}.
The algebraic project remains open for $g \ge 2$.
\hfill\qed
\end{remark}

In outline, the article has the following structure.

In Section~\ref{sec:Z1}, we fix  notation and
 define the Zieschang groupoid, essentially as in~\cite[Section~5.2]{ZVC}
(developed from~\cite{Z64H},~\cite{Z65H},~\cite{Z66a}),
but with modifications taken from work of
McCool~\cite[Lemma~3.2]{DF2}.
We give a simplified proof of a strengthened form of
(the orientable, torsion-free case of)
 Zieschang's result that the  Nielsen-automorphism edges
and the Artin-automorphism edges together generate the
 groupoid.
 Zieschang   used  group\d1theoretical techniques of
Nielsen~\cite{Nielsen1} and Artin~\cite{A},  while McCool used
group\d1theoretical techniques of  Whitehead~\cite{W}.  We  use all  of these.

In Section~\ref{sec:mc}, which is inspired by the proof by
 McCool~\cite{Mc} of the case $p=0$, we define the canonical edges
in the Zieschang groupoid
 and use them to find a special generating set for $\operatorname{\normalfont A}_{g,p}$.

In Section~\ref{sec:main}, we  observe that the results of the previous two sections
immediately imply that the ADL set  generates $\operatorname{\normalfont A}_{g,p}$.
We then present an algebraic translation of Humphries'  proof that the ADLH set also
generates $\operatorname{\normalfont A}_{g,p}$.

At this stage, we will have completed our objective.  For completeness,
we conclude the article with an elementary review of algebraic descriptions of certain mapping-class groups.

In Section~\ref{sec:mcg}, we review  definitions of some mapping-class groups and
mention some of the history of the original proof that the ADLH set generates $\operatorname{\normalfont A}_{g,p}$.

In Section~\ref{sec:boundary},  we recall the definitions of
Dehn twists and braid twists, and see  that the group
$\operatorname{\normalfont A}_{g,p}$ can be viewed as the mapping-class group of
the orientable surface of genus $g$ with $p$ punctures and
one boundary component.

In Section~\ref{sec:noboundary},  we
 consider an orientable surface
 $\mathbf{S}_{g,0,p}$   of genus $g$ with  $p$ punctures, with  $ (g,p) \ne (0,0),\, (0,1)$.
The
\textit{algebraic}
 mapping\d1class group of $\mathbf{S}_{g,0,p}$,  denoted $\operatorname{\normalfont M}_{g,0,p}^{\text{alg}}$,  is defined as
the group of all those outer automorphisms of
$$\pi_1(\mathbf{S}_{g,0,p})= \gp{t_{[1 \uparrow p ]}  \cup x_{[1 \uparrow g ]} \cup  y_{[1 \uparrow g ]}}
{\textstyle \hskip-3pt \operatornamewithlimits{\Pi}\limits_{j \in [p\downarrow 1]} t_j \hskip-2pt
\operatornamewithlimits{\Pi}\limits_{i \in [1 \uparrow g]} [x_i,y_i]} \vspace{-1mm}$$
which permute  the set of conjugacy classes \mbox{$\{\, [t_j], [\overline t_j] : j \in [1{\uparrow}p]\}$.}
We review  Zieschang's  algebraic proof~\cite[Theorem~5.6.1]{ZVC} of Nielsen's result~\cite{Nielsen2} that
$\operatorname{\normalfont M}_{g,0,p}^{\text{alg}}$\vspace{0.5mm} is generated by the natural image of
  $\operatorname{\normalfont A}_{g,p}$\vspace{0.5mm} together with
an   outer automorphism~$\breve \zeta$.  Hence,  $\operatorname{\normalfont M}_{g,0,p}^{\text{alg}}$\vspace{0.5mm}
is generated by the natural image of the ADLH set together with~ $\breve \zeta$.
 In 1966, Magnus, Karrass and Solitar~\cite[p.175]{MKS} remarked that
if one could find a generating set of $\operatorname{\normalfont M}_{g,0,p}^{\text{alg}}$
and  self-homeo\-morph\-isms of $\mathbf{S}_{g,0,p}$ that induce those generators,
then one
would be able to prove that  $\operatorname{\normalfont M}_{g,0,p}^{\text{alg}}$\vspace{0.5mm}
was equal to the (topological)   mapping-class group
$\operatorname{\normalfont M}_{g,0,p}^{\text{top}}$,
even in the then-unknown case  where
 $g \ge 2$ and~$p \ge 2$.
Also in 1966, Zieschang~\cite[Satz~4]{Z66a} used groupoids to prove
 equality, and their remark does not seem to have been followed up.  The generating set
given above fulfills their requirement, since
the image of each  ADL generator   is induced
by a  braid twist  or a Dehn twist of $\mathbf{S}_{g,0,p}$,
and $\breve\zeta$ is induced by a reflection
of~$\mathbf{S}_{g,0,p}$.  This gives a new  way of seeing
that  $\operatorname{\normalfont M}_{g,0,p}^{\text{top}} = \operatorname{\normalfont M}_{g,0,p}^{\text{alg}}$.

\section{The Zieschang groupoid and the Nielsen subgraph}\label{sec:Z1}

In this section, which is  based on~\cite[Section~5.2]{ZVC},
 we define the Zieschang groupoid
$\mathcal{Z}_{g,p}$ and the Nielsen  subgraph $\mathcal{N}_{g,p}$,
and prove that  $\mathcal{N}_{g,p}$  generates   $\mathcal{Z}_{g,p}$.

\medskip

\begin{notation}
We will find it useful to have notation  for intervals in $\integers$
that is different from the notation for intervals in $\reals$.
Let $i$, $j \in \integers$.
We define the sequence $$\bbl1 i{\uparrow}j\bbr1 \coloneq
 \begin{cases}
(i,i+1,\ldots, j-1, j) \in \integers^{j-i+1}   &\text{if $i \le j$,}\\
() \in \integers^0 &\text{if $i > j$.}
\end{cases}
$$
The subset of $\integers$ underlying $\bbl1 i{\uparrow}j\bbr1 $ is denoted $[i{\uparrow}j]\coloneq \{i,i+1,\ldots, j-1, j \}$.

Also,
$[i{\uparrow}\infty[ \,\,\, \coloneq \{i,i+1,i+2,\ldots\}$ .

We define $\bbl1 j{\downarrow}i\bbr1 $ to be the reverse of the sequence $\bbl1 i{\uparrow}j \bbr1 $,
that is, $(j,j-1,\ldots,i+1,i)$.

Suppose that we have a set $X$ and a map $[i{\uparrow}j] \to X$, $\ell \mapsto x_\ell$.
We define the corresponding sequence in $X$ as $$x_{\bbl1 i{\uparrow}j \bbr1 }\coloneq  \begin{cases}
(x_i,x_{i+1},   \cdots, x_{j-1},   x_j) \in  X^{i-j+1}  &\text{if $i \le j$,}\\
()  &\text{if $i > j$.}
\end{cases}$$
By abuse of notation, we shall also express this sequence   as $(x_\ell \mid \ell \in \bbl1 i{\uparrow}j\bbr1 )$,
although \mbox{``$\ell \in \bbl1 i{\uparrow}j\bbr1 $"} on its own will not be assigned a meaning.
The set of terms of~$x_{\bbl1 i{\uparrow}j \bbr1 }$ is denoted~$x_{[i{\uparrow}j]}$.
We define $x_{\bbl1 j{\downarrow}i\bbr1 }$ to be the reverse of the sequence~$x_{\bbl1 i{\uparrow}j\bbr1 }$.
\hfill\qed
\end{notation}

 \begin{notation}  Let $G$ be a multiplicative group.

For each $u \in G$, we denote the inverse of $u$ by both
$u^{-1}$ and $\overline u$.  For $u$, $v \in G$, we let $u^v \coloneq \overline v u v$ and
$[u,v] \coloneq \overline u \, \overline v u v$.
 For $u  \in G$, we let
 $[u] \coloneq \{u^v \mid v \in G\}$,  called the \textit{$G$-conjugacy class of~$u$}.
We let $G/{\sim} \coloneq \{[u] : u \in G\}$, the set of all $G$-conjugacy classes.

Where $G$ is a free group given with a distinguished basis $\mathfrak{B}$,
we think of each $u \in G$ as a reduced word in $\mathfrak{B} \cup  \mathfrak{B}^{-1}$, and
let $\vert u \vert$ denote the
length of the word.
We think of $[u]$ as a cyclically-reduced cyclic word in $\mathfrak{B} \cup  \mathfrak{B}^{-1}$.

Suppose that we have $i$, $j \in \integers$ and a  map  $[i{\uparrow}j] \to G$,
$\ell \mapsto u_\ell$.
We write
\begin{align*}&\operatornamewithlimits{\Pi}_{\ell\in\bbl1 i{\uparrow}j\bbr1 } u_\ell
\coloneq
\Pi u_{\bbl1 i{\uparrow}j\bbr1 }\coloneq     \begin{cases}
u_i   u_{i+1}   \cdots u_{j-1}  u_j \in G   &\text{if $i \le j$,}\\
1 \in G &\text{if $i > j$.}
\end{cases}\\ &\operatornamewithlimits{\Pi}_{\ell\in\bbl1 j{\downarrow}i\bbr1 } u_\ell
\coloneq \Pi u_{\bbl1 j{\downarrow}i\bbr1 }  \coloneq  \begin{cases}
u_j   u_{j-1}   \cdots u_{i+1}   u_i \in G  &\text{if $j \ge i$,}\\
1 \in G &\text{if $j< i $.}
\end{cases}\end{align*}

When we have  $G$ acting on a set $X$, then, for each $x \in X$, we let $\Stab(x; G)$ denote
the set of elements of $G$ which stabilize, or fix,  $x$.

We let $\Aut G$ denote the group of all automorphisms of $G$,
 acting on the right, as exponents,  $u \mapsto u^\phi$.
In a natural way, $\Aut G$ acts on $G/{\sim}$ and on the set of subsets of $G \cup (G/{\sim})$.

We let $\Out G$ denote the quotient of  $\Aut G$ modulo the group of inner automorphisms,
  we call the elements of $\Out G$ \textit{outer automorphisms}, and
we denote the quotient map  $\Aut G \to \Out G$ by $\phi \mapsto \breve \phi$.
In a natural way, $\Out G$ acts on $G/{\sim}$ and on the set of subsets of $ G/{\sim} $.
 \hfill \qed
\end{notation}

\begin{notation} The following will be fixed throughout.

Let $g$, $p \in [0{\uparrow}\infty[$\,.
Let $F_{g,p} \coloneq \gp{t_{[1 \uparrow p]} \cup x_{[1 \uparrow g ]} \cup  y_{[1 \uparrow g ]}}{\quad}$,
a free group of rank $2g{+}p$ with a distinguished basis.
We shall find it convenient to use   abbreviations such as
$$ [ t]_{[1 {\uparrow} p]} \coloneq \{ [ t_j] : j \in [1 {\uparrow} p]\} ,  \qquad \hskip-16pt t_{[1{\uparrow}p]}^{\pm 1} \coloneq
\{t_j, \overline t_j  :  j \in [1 {\uparrow} p]\}, \qquad \hskip-16pt
\operatornamewithlimits{\Pi}  [x ,y ]_{\bbl1 1 {\uparrow} g \bbr1 } \coloneq
\operatornamewithlimits{\Pi}\limits_{i \in \bbl1 1 \uparrow g \bbr1 } [x_i,y_i].$$

The elements of
$ t_{[1{\uparrow}p]}^{\pm 1}
\cup   x_{[1{\uparrow}g]}^{\pm 1}   \cup
y_{[1{\uparrow}g]}^{\pm 1}  $ will be called \textit{letters}.
The elements of $t_{[1{\uparrow}p]} $
will be called \textit{$t$-letters}.
The elements of $\overline t_{[1{\uparrow}p]} $
will be called \textit{inverse $t$-letters}.
The elements of $x_{[1{\uparrow}g]}^{\pm 1}   \cup
y_{[1{\uparrow}g]}^{\pm 1}$
will be called \textit{$x$-letters}.

We shall usually codify an element $\phi \in \Aut F_{g,p}$ as a two-row matrix where the
first row gives, for some basis consisting of letters,  all those  elements  which are moved by $\phi$,
and the second row  equals the $\phi$-image of the first row.

We shall be working throughout with the group $\Stab([ t]_{[1 {\uparrow} p]} ; \Aut F_{g,p})$ (which permutes the
set of cyclic words $[ t]_{[1 {\uparrow} p]}$)  and its subgroup
  $$\operatorname{\normalfont A}_{g,p} \coloneq
 \Stab( [ t]_{[1 {\uparrow} p]} \cup \{\operatornamewithlimits{\Pi}  t_{\bbl1 p {\downarrow} 1 \bbr1 }
 \operatornamewithlimits{\Pi}  [x ,y ]_{\bbl1 1 {\uparrow} g \bbr1 } \} ;\Aut F_{g,p}).$$
\vskip -6mm \hfill\qed
\end{notation}

\medskip

\begin{definitions} Let $g$, $p \in [0{\uparrow}\infty[$\, and let
$F_{g,p} \coloneq \gp{t_{[1 \uparrow p]} \cup x_{[1 \uparrow g ]} \cup  y_{[1 \uparrow g ]}}{\quad}$.

Let
  $  [1{\uparrow}(4g{+}p)] \to
 t_{[1{\uparrow}p]} \cup x_{[1{\uparrow}g]}^{\pm 1} \cup  y_{[1{\uparrow}g]}^{\pm 1}$,\,\, $k \mapsto v_k,$
be a bijective map, let \mbox{$V \coloneq \operatornamewithlimits{\Pi}\limits v_{\bbl1 1{\uparrow}(4g+p)\bbr1 }$},
and let $\Gamma$ denote the graph with
\begin{list} {}{}
\item  vertex set
 $ t_{[1{\uparrow}p]}^{\pm 1} \cup x_{[1{\uparrow}g]}^{\pm 1} \cup   y_{[1{\uparrow}g]}^{\pm 1} $, and
\item  edge set   $\{(\overline t_j {\rightsquigarrow} t_j) \mid j \in [1{\uparrow}p]\}
\,\, \cup\,\,
\{(v_k {\rightsquigarrow} \overline v_{k+1}) \mid k \in [1{\uparrow}(4g{+}p{-}1)]\}.$
\end{list}
If $\Gamma$ has no cycles (that is,   $\Gamma$ is a forest), then
we  say that
 $V$  is a
 \textit{Zieschang element} of~$F_{g,p}$ and that  $\Gamma$ is   the
 \textit{extended Whitehead graph of~$V$}; we note that the condition that $\Gamma$ has no cycles implies that
 $ \operatornamewithlimits{\Pi}\limits v_{\bbl1 1{\uparrow}(4g+p)\bbr1 }$
 is  the  reduced expression for $V$, and, hence, $\Gamma$ is the usual Whitehead graph of
 \mbox{$[\overline   t]_{[1 {\uparrow} p]} \cup
\{  V  \}$}, as in~\cite{W}.
If $(g,p) \ne (0,0)$ and $V$ is a Zieschang element of~$F_{g,p}$,
 then $\Gamma$ has the form of an oriented line segment with $4g{+}2p$ vertices and
$4g{+}2p{-}1$ edges; here, we
  define \mbox{$v_0 \hskip-1pt\coloneq \hskip-1pt v_{4g+ p+1} \hskip-1pt \coloneq  \hskip-1pt 1$}, and
book-end $\Gamma$  with the \textit{ghost} edges
$(v_0 {\rightsquigarrow} \overline v_1)$ and  $(v_{4g+ p} {\rightsquigarrow} \overline v_{4g+ p+1})$.

For example,  $V_0
\coloneq \Pi t_{\bbl1 p{\downarrow}1\bbr1 }\Pi [x ,y ]_{\bbl1 1{\uparrow}g \bbr1 }$ is a Zieschang element of
$F_{g,p}$, and its
extended Whitehead graph is  $$ \overline t_p  {\rightsquigarrow} t_p {\rightsquigarrow} \overline t_{p-1} {\rightsquigarrow} t_{p-1} {\rightsquigarrow} \cdots {\rightsquigarrow} \overline t_1 {\rightsquigarrow} t_1
{\rightsquigarrow} x_1 {\rightsquigarrow} \overline y_1 {\rightsquigarrow} \overline x_1 {\rightsquigarrow} y_1
{\rightsquigarrow} x_2 {\rightsquigarrow}  \cdots
{\rightsquigarrow} x_g {\rightsquigarrow} \overline y_g {\rightsquigarrow} \overline x_g {\rightsquigarrow} y_g .$$

The   \textit{Zieschang groupoid} for $F_{g,p}$, denoted $ \mathcal{Z}_{g,p}$,
is defined as  follows.
 \begin{list}{$\bullet$}{}
\item The set $\operatorname{V}\!\! \mathcal{Z}_{g,p}$ of  vertices/objects   of $ \mathcal{Z}_{g,p}$
equals the set of  Zieschang elements of $F_{g,p}$.
\item The edges/elements/morphisms    of $\mathcal{Z}_{g,p}$ are the triples $(V,W,\phi)$
such that  \mbox{$V$, $W  \in   \operatorname{V}\!\! \mathcal{Z}_{g,p}$},
and
\mbox{$\phi \in  \Stab([ t]_{[1 {\uparrow} p]} ; \Aut F_{g,p})$,}
and $V^\phi = W$.  Here, we say that $(V {\xrightarrow{\phi}} W)$, or
 $V {\xrightarrow{\phi}} W$,    is an edge of $ \mathcal{Z}_{g,p}$  \textit{from $V$ to~$W$},
and denote the set of such edges by $\mathcal{Z}_{g,p}(V,W)$.

\item The partial multiplication in $\mathcal{Z}_{g,p}$  is defined using the multiplication in
$\Stab([ t]_{[1 {\uparrow} p]} ; \Aut F_{g,p})$
in the natural way.
\end{list}

\medskip

If $V  \in \operatorname{V}\!\! \mathcal{Z}_{g,p}$,  then,  as a group,
\mbox{$\mathcal{Z}_{g,p}(V,V) = \Stab([ t]_{[1 {\uparrow} p]} \cup \{V\}; \Aut F_{g,p}).$}
Thus
$\mathcal{Z}_{ g,p }(V_0, V_0) = \operatorname{\normalfont A}_{g,p}$.  Throughout, we shall view the elements of
$\operatorname{\normalfont A}_{g,p}$ as edges of $\mathcal{Z}_{g,p}$ from $V_0$ to $V_0$.
We shall be using $V_0$ as a basepoint of $\mathcal{Z}_{g,p}$ in Definitions~\ref{defs:phi},
where we will verify  that $\mathcal{Z}_{g,p}$ is connected.  \hfill\qed
\end{definitions}

\begin{definitions}  Let   $g$, $p \in [0{\uparrow}\infty[$\,,
let $F_{g,p} \coloneq \gp{t_{[1 \uparrow p]} \cup x_{[1 \uparrow g ]} \cup  y_{[1 \uparrow g ]}}{\quad}$\vspace{.5mm},
let  \mbox{$V$\!, $W \in F_{g,p}$,} and let  $\phi \in   \Aut F_{g,p} $.
Suppose that $V \in \operatorname{V}\!\!\mathcal{Z}_{g,p}$,
and that  $ V^\phi = W$.

 If $\phi$ permutes
the $t$-letters and  permutes the $x$-letters, then
 we say
that $V {\xrightarrow{\phi}} W$ is a \textit{$\iN_1$ edge  in $\mathcal{Z}_{g,p}$}.
To see that $(V {\xrightarrow{\phi}} W) \in \mathcal{Z}_{g,p}$,
notice that \mbox{$\phi \in \Stab([ t]_{[1 {\uparrow} p]} ; \Aut F_{g,p})$}
and  $W \in \operatorname{V}\!\!\mathcal{Z}_{g,p}$.

If there exists
 some \mbox{$k \in [1{\uparrow}(4g{+}p{-}1)]$ } such that the letter
\mbox{$v_k$} is an $x$-letter and
$\phi = \left(\begin{smallmatrix}
 v_k\phantom{v_{k+1}}  \\ v_k \overline v_{k+1}
\end{smallmatrix}\right)$, then
we say that $V {\xrightarrow{\phi}} W$ is a \textit{right $\iN_2$ edge  in $\mathcal{Z}_{g,p}$.}
To see that $(V {\xrightarrow{\phi}} W) \in \mathcal{Z}_{g,p}$, we note the following.
  In passing from $V$ to $W$,
we remove the boxed part in $v_k\boxed{\!v_{k+1}\!\!}\,v_{k+2}$  and  add the boxed part in
$v_{j-1} \boxed{\!v_{k+1}\!\!}\, v_j$,  where  $v_j = \overline v_k$.
In passing from the extended Whitehead graph of $V$ to the extended Whitehead graph of $W$,
we remove the boxed part  in   $v_{j-1}\!{\rightsquigarrow}\boxed{\!\overline v_j{=}v_k{\rightsquigarrow}\!\!} \,  \overline v_{k+1}$
and add the boxed part in \mbox{$v_{k+1}\!{\rightsquigarrow}\boxed{\!\overline v_j{=}v_k{\rightsquigarrow}\!\!}\,\, \overline v_{k+2}$},
where we have indicated a
 ghost edge  if  $j=1$ or $k= 4g{+}p{-}1$. Hence, $(V {\xrightarrow{\phi}} W) \in \mathcal{Z}_{g,p}$.

If there exists some $k \in [2{\uparrow}(4g{+}p)]$ such that
\mbox{$v_k$} is an $x$-letter and
$\phi = \left(\begin{smallmatrix}
v_k\phantom{v_{k-1}}\\  \overline v_{k-1} v_k
\end{smallmatrix}\right)$,
then we say that $V {\xrightarrow{\phi}} W$ is a \textit{left $\iN_2$ edge  in $\mathcal{Z}_{g,p}$}.
This is an inverse of an edge of the previous type.

By a   \textit{$ \iN_2$  edge  in $\mathcal{Z}_{g,p}$}, we mean a left    or
  right $\N_2$ edge  in $\mathcal{Z}_{g,p}$.

If there exists some \mbox{$k \in [1{\uparrow}(4g{+}p{-}1)]$} such that
the letter $v_k $ is a $t$-letter and
$\phi = \left(\begin{smallmatrix}
v_k\phantom{v_{k+1}v_{k+1}} \\ v_{k+1} v_k \overline v_{k+1}
\end{smallmatrix}\right)$, then we say that $V {\xrightarrow{\phi}} W$ is a
\textit{right $\iN_3$ edge  in $\mathcal{Z}_{g,p}$}.
To see that $(V {\xrightarrow{\phi}} W) \in \mathcal{Z}_{g,p}$, we note the following.
 In passing from $V$ to $W$,
we change $v_{k-1}v_k\boxed{\!v_{k+1}\!\!}v_{k+2}$ to $v_{k-1}\boxed{\!v_{k+1}\!\!}\,v_kv_{k+2}$.
In passing from the extended Whitehead graph of $V$ to the extended
Whitehead graph of $W$,
we remove the  boxed part  in
$v_{k-1}\!{\rightsquigarrow}\boxed{\!\overline v_k{\rightsquigarrow}v_k{\rightsquigarrow}\!\!}\,\,  \overline v_{k+1}$
and add the boxed part in
$v_{k+1}\!{\rightsquigarrow}\boxed{\!\overline v_k{\rightsquigarrow}v_k{\rightsquigarrow}\!\!}\,\, \overline v_{k+2}$,
where we have indicated  a ghost edge
  if  $k=1$ or \mbox{$k= 4g{+}p{-}1$}. Hence,
\mbox{$(V {\xrightarrow{\phi}}W) \in \mathcal{Z}_{g,p}$}.

If there exists some $k \in [2{\uparrow}(4g{+}p)]$ such that the letter
$v_k $ is a $t$-letter and
$\phi = \left(\begin{smallmatrix}
v_k\phantom{v_{k-1}v_{k-1}} \\ \overline  v_{k-1} v_k v_{k-1}
\end{smallmatrix}\right)$,
we say that $V {\xrightarrow{\phi}} W$ is a \textit{left $\iN_3$ edge  in $\mathcal{Z}_{g,p}$}.
This is an inverse of an edge of the previous type.

By a   \textit{$ \iN_3$  edge  in $\mathcal{Z}_{g,p}$}, we mean a left    or
  right $\N_3$ edge  in $\mathcal{Z}_{g,p}$.

By   a \textit{Nielsen edge  in $\mathcal{Z}_{g,p}$}, we mean
  a $\N_i$ edge  in $\mathcal{Z}_{g,p}$, for some $i \in \{1,2,3\}$.

We define the \textit{Nielsen  subgraph
  of $\mathcal{Z}_{g,p}$}, denoted $\mathcal{N}_{g,p}$, to be the graph with vertex set
\mbox{$ \operatorname{V}\!\!\mathcal{Z}_{g,p}$}
 and edges, or elements, the  Nielsen edges in $\mathcal{Z}_{g,p}$.
\hfill\qed
\end{definitions}

We now give a simplified proof of a result  due to Zieschang and McCool.

\begin{theorem}\label{thm:z}  Let $g$, $p \in [0{\uparrow}\infty[$\,,
let $F_{g,p} \coloneq \gp{t_{[1 \uparrow p]} \cup x_{[1 \uparrow g ]} \cup  y_{[1 \uparrow g ]}}{\quad}$,
let  \mbox{$V$\!, $W \in F_{g,p}$,}
 let $H$ be a free group,  let $\phi $ be an  endomorphism of $H{\ast}F_{g, p }$, and
 suppose that  the following hold.\vspace{-.5mm}
\begin{enumerate}[{\normalfont (a).}]
\item $V \in  \operatorname{V}\!\!\mathcal{Z}_{g,p}$.
\item $\vert W \vert \le 4g+p$.
\item $V^\phi = W$.
\item There exists some permutation $\pi$  of $[1{\uparrow}p]$ such that,
for each \mbox{$j \in [1{\uparrow}p]$}, $t_j^\phi$ is   $(H{\ast}F_{g, p })$-conjugate to~$t_{j^\pi}$.
\item  $F_{g, p }^\phi \simeq F_{g, p }$.  \vspace{-2mm}
\end{enumerate}
Then $W  \in  \operatorname{V}\!\!\mathcal{Z}_{g,p}$  and
there  exists an edge $V {\xrightarrow{\phi'}}W$    in the subgroupoid of
$\mathcal{Z}_{ g,p }$ generated by the Nielsen subgraph $\mathcal{N}_{g,p}$ such that
$\phi$ acts  as~$\phi'$ on  the free factor~$F_{g,p}$.
\end{theorem}

\begin{proof}  We may assume that $(g,p) \ne (0,0)$.
Extend $  t_{[1{\uparrow}p]} \cup  x_{[1{\uparrow}g]} \cup  y_{[1{\uparrow}g]}$ to
  a basis $\mathfrak{B}$ of the free group $F_{g,p}{\ast} H$.  For each
  \mbox{$A \in F_{g,p}{\ast} H$},   $\vert A \vert$
denotes the   length of
 $A$  as a reduced product in $\mathfrak{B}{\cup}\mathfrak{B}^{-1}$.
Choose a total order, denoted $\le$,  on $\mathfrak{B}{\cup}\mathfrak{B}^{-1}$, and
extend  $\le$ to a length-lexicographic total order, also denoted $\le$, on $F_{g,p}{\ast} H$.

Consider the reduced expression  $V =  \Pi  v_{\bbl1 1{\uparrow}(4g+p)\bbr1 }$.
Let $v_0 \coloneq v_{4g+p+1} \coloneq 1$.

For each $k \in [0{\uparrow}(4g{+}p)]$, let
 $A_k$ denote the largest common initial subword of $\overline v_k^\phi$ and $ v_{k+1}^\phi$
with respect to $\mathfrak{B}{\cup}\mathfrak{B}^{-1}$.
Since   $v_0 = v_{4g+p+1} = 1$, we have \mbox{$A_0 = A_{4g+p} = 1$}.
For each $k \in [1{\uparrow}(4g{+}p)]$, let $w_k \coloneq  \overline A_{k-1} v_k^\phi   A_{k} \in F_{g,p}{\ast} H$.
Then \mbox{$v_k^\phi  =  A_{k-1} w_k \overline A_{k}$,} where this expression need not be reduced.

We shall  show in Claim 1  that we may assume that  \mbox{$A_k < A_{k-1} w_k$} and that
\mbox{$A_{k} < A_{k+1} \overline w_{k+1}$},
and then show in Claim 2 that this ensures that $\phi$ permutes
the $t$-letters and permutes the $x$-letters.

We let   $\binom{F_{g,p}{\ast} H}{4g{+}p}$ denote the set of $(4g{+}p)$-element subsets of   $F_{g,p}{\ast} H$, and
  define a pre-order $\preccurlyeq$ on $\binom{F_{g,p}{\ast} H}{4g{+}p}$ as follows.
For each $A \in F_{g,p}{\ast} H$, there is a unique reduced
expression $A = A^{(L)}A^{(R)}$ with the property that
\mbox{$ \vert A^{(L)} \vert  -  \vert A^{(R)}\vert  \in \{0,1\}$}.
For $A, B \in F_{g,p}{\ast} H$, we write $A \preccurlyeq B$ if either
$\vert A\vert < \vert B \vert$ or  ($\vert A\vert = \vert B \vert$
and \mbox{$A^{(L)} \le B^{(L)}$}).
We can arrange each element of $\binom{F_{g,p}{\ast} H}{4g{+}p}$ as a (not necessarily
unique) ascending sequence with respect to $\preccurlyeq$,
and assign  $\binom{F_{g,p}{\ast} H}{4g{+}p}$ the (unique) lexicographic
pre-order, again denoted $\preccurlyeq$.  Here, $\mathbb{A}\prec\mathbb{B}$ will mean
 $\mathbb{A}\preccurlyeq\mathbb{B}$ and  $\mathbb{B}\not\preccurlyeq\mathbb{A}$.

Without assigning any meaning to $V {\xrightarrow{\phi}} W$, let us write
$$\mu(V {\xrightarrow{\phi}} W)
 \coloneq    t_{[1{\uparrow}p]}^{ \phi} \cup
 (x_{[1{\uparrow}g]}^{\pm 1})^{ \phi}\cup  (y_{[1{\uparrow}g]}^{\pm 1})^{ \phi}  \,\,\,\in\,\,\,
(\textstyle\binom{F_{g,p}{\ast} H}{4g{+}p}, \preccurlyeq).$$
 It follows from (e) that  there are $4g{+}p$ distinct elements in the set $\mu(V {\xrightarrow{\phi}}W)$.

\medskip

\noindent \textbf{Claim 1}. \textit{Let $k \in [1{\uparrow}(4g{+}p{-}1)]$.
If    $A_k \ge A_{k-1} w_k$ or $A_{k} \ge A_{k+1} \overline w_{k+1}$,\vspace{-.1mm} then
there exists some
$(V {\xrightarrow{\alpha}} U) \in  \mathcal{N}_{g,p}$
such that $\mu(U {\xrightarrow{\overline \alpha\phi}} W)
\prec \mu(V{\xrightarrow{\phi}}W)$.}

\begin{proof} [Proof of Claim $1$]  We have specified reduced expressions
$ v_k^\phi = B \overline  A $ and \mbox{$ v_{k+1}^\phi =  A  \overline C$} and
$ v_k^\phi v_{k+1}^\phi = B\overline C$, where
  $A \coloneq A_k$,   $B \coloneq A_{k-1}  w_k $,    \mbox{$C \coloneq  A_{k+1} \overline w_{k+1}$}.
  It follows from~(e) that
 $A$, $B$, and $C$ are all different.

 By hypothesis,
\mbox{$A \ne \min(\{A,B,C\}, \le)$}.  We shall consider only the case where \mbox{$B = \min(\{A,B,C\}, \le)$};
the argument where $C = \min(\{A,B,C\}, \le) $ is similar.  Thus we have $A > B < C$.

The letter $v_{k+1}$ is either a $t$-letter or an $x$-letter.

\medskip

\noindent \textbf{Case 1.} $v_{k+1}$ is an $x$-letter.

On taking $\alpha \coloneq \left( \begin{smallmatrix}
v_{k+1} \\ \overline v_k v_{k+1}
\end{smallmatrix}\right) $, we have a    $\N_2$ edge
$(V {\xrightarrow{\alpha}}U) \in \mathcal{N}_{g,p}$.
Here \mbox{$\overline \alpha \coloneq \left( \begin{smallmatrix}
v_{k+1} \\  v_k v_{k+1}
\end{smallmatrix}\right) $} and $v_{k+1}^{\overline \alpha\phi} = v_k^\phi v_{k+1}^\phi = B\overline C$.
In this case, the change from $\mu(V {\xrightarrow{\phi}}W)$ to
$\mu(U {\xrightarrow{\overline \alpha \phi}}W)$
consists of replacing
$\{v_{k+1}^\phi, \overline v_{k+1}^\phi\}  = \{A  \overline C, C \overline A\} $ with
$\{v_{k+1}^{\overline \alpha \phi}, \overline v_{k+1}^{\overline \alpha \phi}\}  = \{B\overline C, C \overline B\} $.
To show that $\mu(U {\xrightarrow{\overline \alpha\phi}} W)
\prec \mu(V {\xrightarrow{\phi}} W)$, it now suffices to show that
$B \overline C \prec A  \overline C$ and $C \overline B \preccurlyeq C \overline A$.

If \mbox{$\vert A\vert > \vert B \vert$,} then
\mbox{$\vert B\overline C \vert = \vert C \overline B \vert = \vert B \vert + \vert C \vert <
\vert A \vert + \vert C \vert = \vert A  \overline C\vert = \vert C \overline A\vert$}, and,\linebreak hence,
$B \overline C \prec A  \overline C$ and $C \overline B \prec   C \overline A$.

If $\vert A\vert = \vert B \vert$, then, since $ A > B < C$, we have $\vert A\vert =\vert B\vert \le \vert C \vert$
and $B < A$.  Hence   $B \overline C \prec A  \overline C$ and $C \overline B \preccurlyeq C \overline A$.

\noindent \textbf{Case  2.}  $v_{k+1}$ is a $t$-letter.

On taking $\alpha \coloneq \left( \begin{smallmatrix}
v_{k+1} \\ \overline v_k v_{k+1} w_k
\end{smallmatrix}\right) $, we have a $\N_3$ edge
$(V {\xrightarrow{\alpha}} U) \in\mathcal{N}_{g,p}$.
Here \mbox{$\overline \alpha \coloneq \left( \begin{smallmatrix}
v_{k+1} \\  v_k v_{k+1} \overline v_k
\end{smallmatrix}\right)$.}
In this case, the change from $\mu(V {\xrightarrow{\phi}} W)$ to
\mbox{$\mu(U {\xrightarrow{\overline \alpha \phi}} W)$},
consists of replacing
\mbox{$v_{k+1}^\phi   =  A\overline C$} with
$ v_{k+1}^{\overline \alpha \phi}  = v_k^\phi v_{k+1}^\phi \overline v_k^\phi  =
(B\overline A)(A\overline C)(A \overline B)=  B\overline C A \overline B $.
To show that \mbox{$\mu(U {\xrightarrow{\overline \alpha\phi}}W)
\prec \mu(V {\xrightarrow{\phi}} W)$}, it suffices to show that
$ B\overline C A \overline B \prec A  \overline C$.

Let $D \coloneq \min(\{A,C\}, \le )$.  Since  $v_{k+1}$ is a $t$-letter,
there exists some $j \in [1{\uparrow}p]$ such that
$ v_{k+1}^\phi $    is a conjugate of   $t_j$, that is,  $A \overline C $
is a conjugate of   $t_j$.
Thus, both $A \overline C $ and $C \overline A$ begin with $D$,
and we can write
$A \overline C = DE t_j \overline E\, \overline D$ with no cancellation.   Now
 $ E t_j \overline E = \overline D  A \overline C D = \overline C A$.
Hence $B\overline C A \overline B = BEt_j \overline E\, \overline B$ where
this expression may have cancellation.  Recall that  $B < D$.
Thus $ BEt_j \overline E\, \overline B
\prec DE t_j \overline E\, \overline D$, that is,
$B\overline C A \overline B  \prec A \overline C$.

\medskip

This completes the proof of Claim 1.
\end{proof}

Claim 1 gives a procedure for reducing $\mu(V {\xrightarrow{\phi}}W)$.
Once $\phi$ is specified, only a finite subset of
$\mathfrak{B}{\cup}\mathfrak{B}^{-1}$ is ever involved, and, moreover, there is
an upper bound for the  lengths of the elements of
$F_{g,p}{\ast}H$ which will appear.  It follows that we can repeat
the procedure  only a finite number of times.
Hence, we may now assume that,  for each $k \in [1{\uparrow}(4g{+}p{-}1)]$,
$A_k < A_{k-1} w_k$ and \mbox{$A_{k} < A_{k+1} \overline w_{k+1}$}.

\medskip

\noindent \textbf{Claim 2.}  \textit{Under the latter assumption,
$\phi$ permutes the $t$-letters and permutes the $x$-letters, and the desired conclusion holds.}

\bigskip

\noindent \textit{Proof of Claim $2$.}
For each $k \in [1{\uparrow}(4g{+}p)]$, $A_k < A_{k-1} w_k$
and $A_{k-1} < A_{k} \overline w_{k}$ (even for $k = 1$ and $k = 4g{+}p$).
It follows that $ w_k  \ne 1$ and also that  the expression
\mbox{$v_k^\phi  =  A_{k-1} w_k \overline A_{k}$}  is reduced.
It then follows that, for each $k \in  [1{\uparrow}(4g{+}p{-}1)]$,
\mbox{$A_{k-1} w_k w_{k+1} \overline A_{k+1}$} is a reduced expression for $v_k^\phi v_{k+1}^\phi$.
Now
\begin{equation*}
W=V^\phi = \textstyle( \Pi  v_{\bbl1 1{\uparrow}
(4g+p) \bbr1 })^\phi   = \mkern-10mu
\textstyle\operatornamewithlimits{\Pi}\limits_{\null_{k \in \bbl1 1{\uparrow} (4g+p) \bbr1 }}
 \mkern-30mu( A_{k-1} w_k \overline A_{k})=
\textstyle\Pi  w_{\bbl1 1{\uparrow} (4g+p) \bbr1 },
\end{equation*}
 and we have just seen that the expression $\Pi  w_{\bbl1 1{\uparrow} (4g+p) \bbr1 }$ is  reduced.
By (b), $$\textstyle 4g{+}p \mkern5mu\ge\mkern5mu
\vert  W  \vert =  \vert
\operatornamewithlimits{\Pi}  w_{\bbl1 1{\uparrow} (4g+p) \bbr1 }
\vert =   \textstyle\sum\limits_{k=1}^{4g{+}p}     \vert   w_k  \vert
\mkern5mu \ge\mkern5mu   4g{+}p.$$
Hence, equality holds throughout, and, for each $k \in [1{\uparrow}(4g{+}p)]$, $\vert w_k \vert = 1$
and $w_k$ is a letter.

 Let $s_{\bbl1 1{\uparrow}(4g+2p)\bbr1 }$ be the vertex sequence in the extended Whitehead graph of $V$, that is,
  \mbox{$s_{[1{\uparrow}(4g+2p)]} = t_{[1{\uparrow}  p]}^{\pm 1} \cup   x_{[1{\uparrow}  g]}^{\pm 1} \cup
y_{[1{\uparrow}g]}^{\pm 1}$}    and
$ \{(s_{\ell}{\rightsquigarrow}s_{\ell +1})\mid \ell \in [1{\uparrow} (4g{+}2p{-}1)] \} $ equals
$  \{(\overline t_j{\rightsquigarrow}t_j) \mid j \in [1{\uparrow}p]\}\,\, \cup\,\,
\{(v_k{\rightsquigarrow}\overline v_{k+1})
\mid k \in [1{\uparrow}(4g{+}p{-}1)]\}.$

We assume that there exists some $\ell \in [1{\uparrow}(4g{+}2p)]$ such that
$\vert s_{\ell}^\phi \vert > 1$, and we shall obtain a contradiction.
Let $s_{\ell}^\phi$ end in  $b \in \mathfrak{B}{\cup}\mathfrak{B}^{-1}$.
Assume further that $\ell$ has been chosen to minimize $\overline b$
in  $(\mathfrak{B}{\cup}\mathfrak{B}^{-1}, \le)$.  Assume further that
$\ell$ has been chosen maximal.  In particular, if $\vert s_{\ell+1}^\phi  \vert > 1$,
then $s_{\ell+1}^\phi$ does not end in $b$.

Recall that $(s_{\ell}{\rightsquigarrow}s_{\ell+1})$ can be expressed  either as
 $(\overline t_j{\rightsquigarrow}t_j)$
or as $(v_k{\rightsquigarrow}\overline v_{k+1})$, possibly a ghost edge.
If $(s_{\ell}{\rightsquigarrow}s_{\ell+1}) = (\overline t_j{\rightsquigarrow}t_j)$, then
 $\vert s_{\ell+1}^\phi  \vert =  \vert s_{\ell}^\phi  \vert> 1$, and, also,
$s_{\ell+1}^\phi$ ends in $b$.  This is a contradiction.
Thus, we may assume that
 $(s_{\ell}{\rightsquigarrow}s_{\ell+1}) =  (v_k{\rightsquigarrow}\overline v_{k+1})$,
possibly with $k= 4g{+}p$.  Then $s_{\ell}^\phi =    v_k^\phi  =  A_{k-1} w_k \overline A_k $
and $A_{k-1} <   A_k\overline w_k$.

We claim that $A_k = 1$. Suppose not. Then   $k < 4g{+}p$ and, also, $\overline A_k$ ends in $b$.  Now
$ \overline  s_{\ell+1}^\phi = v_{k+1}^\phi = A_{k} w_{k+1} \overline A_{k+1}$.
Thus $\vert s_{\ell+1}^\phi\vert > 1$ and $s_{\ell+1}^\phi$ ends in $b$.  This is a
contradiction.  Hence $A_k=1$.

Now,  $s_{\ell}^\phi =     A_{k-1} w_k \overline A_k =  A_{k-1} w_k$.  Here, $w_k = b$ and, also,
$A_{k-1} \ne 1$.
Now,  \mbox{$A_{k-1} <   A_k\overline w_k =\overline w_k = \overline b$}.
Thus,  $A_{k-1} \in \mathfrak{B}{\cup}\mathfrak{B}^{-1}$.
Write $a \coloneq \overline  A_{k-1}\in \mathfrak{B}{\cup}\mathfrak{B}^{-1}$.
Then $s_{\ell}^\phi =  \overline a b$ and $\overline a < \overline b$.
There exists some $\ell'\in [1{\uparrow}(4g{+}2p)]$ such that
$s_{\ell'} = \overline s_{\ell}$.  Then $ s_{\ell'}^\phi = \overline b  a$  and $\overline a < \overline b$.
This contradicts the minimality of $\overline b$.

We have now shown that $\phi$  permutes the $t$-letters and maps the $x$-letters to letters.
It follows from (e)  that   $\phi$ permutes the $x$-letters.
Hence, $\phi$
gives a $\N_1$ edge in $\mathcal{N}_{g,p}$.

This completes the proof of Claim 2 and the proof of the theorem.  \qed
\end{proof}

  Theorem~\ref{thm:z}  combines
 Zieschang's approach~\cite[Section~5.2]{ZVC}
and McCool's approach~\cite[Lemma~3.2]{DF2}.
Zieschang does not use  Whitehead graphs explicitly and McCool does not use   $\N_3$ edges explicitly.
For Claim~1, the ingenious pre-order and the proof of Case 1 go  back to Nielsen~\cite{Nielsen1}, and the
proof of Case~2  goes back to Artin~\cite{A}.
The proof of Claim~2 goes back to Whitehead~\cite{W}.
Zieschang  refers to Nielsen~\cite{Nielsen1} for the proof  of his version of Claim 1   and
gives a long proof of his version of Claim 2.
McCool uses results of Whitehead~\cite{W} for the proof of  his version of Theorem~\ref{thm:z}.

We shall be interested in  five  special cases.

In Theorem~\ref{thm:z}, we can take $H = 1$  and take
$(V  {\xrightarrow{\phi}}W) \in \mathcal{Z}_{g,p}$ to see the following.

\begin{consequence}\label{cons:z1}   $\mathcal{Z}_{g,p}$ is generated by  $\mathcal{N}_{g,p}$.
\hfill\qed
\end{consequence}

In Theorem~\ref{thm:z}, we can take $H = 1$
and take $\phi$ to be an automorphism to obtain
the following weak form of
results of   Whitehead.

\begin{consequence}\label{cons:z1b}  For
$V \in  \operatorname{V}\!\!\mathcal{Z}_{g,p}$ and $\phi \in \Stab([t]_{[1{\uparrow}p]}; \Aut F_{g,p})$,
if   $\vert  V^\phi \vert \le 4g{+}p$, then
$ V^\phi \in \operatorname{V}\!\!\mathcal{Z}_{g,p}$.
\hfill\qed
\end{consequence}

It is a classic result of Nielsen~\cite{Nielsen1} that every surjective endomorphism of
a finite-rank free group is an automorphism, and his proof is the basis of the above proof of Claim~1.
A special case of this classic result will be used later in
reviewing a proof of another result of Nielsen, Theorem~\ref{thm:niel},  and to make our
exposition self-contained, we now note that we have proved the desired special case.
We have also proved one of Zieschang's results concerning
injective endomorphisms being automorphisms.

In Theorem~\ref{thm:z}, we can take  $H = 1$ to obtain the following.

\begin{consequence}\label{cons:z2a} Suppose that $\phi$
is an  endomorphism  of
$F_{g,p}$  such that $\phi$ is surjective  or injective, and
such that $\phi$ fixes $\Pi t_{\bbl1 p{\downarrow}1\bbr1 }\Pi [x ,y ]_{\bbl1 1{\uparrow}g \bbr1 }$
and such that there exists some permutation $\pi$  of $[1{\uparrow}p]$ such that,
for each \mbox{$j \in [1{\uparrow}p]$}, $t_j^\phi$ is   $F_{g,p}$-conjugate to~$t_{j^\pi}$.
Then  $\phi$ is an automorphism.
\hfill\qed
 \end{consequence}

\begin{consequence}\label{cons:z2b}   Suppose that $p \ge 1$.

Let us identify $F_{g,p} = H {\ast} F_{g,p-1}$ where
$H \coloneq \gp{t_p}{\quad}$.

 Let
 $V  \coloneq \Pi t_{\bbl1 (p-1){\downarrow}1\bbr1 }\Pi [x ,y ]_{\bbl1 1{\uparrow}g \bbr1 }
\in F_{g,p-1}$ and $\phi \in  \Stab(V; \operatorname{\normalfont A}_{g,p})=\Stab(t_{p}; \operatorname{\normalfont A}_{g,p})$.
By Theorem~\ref{thm:z}, $\phi$ acts as an automorphism $\phi'$ on $F_{g,p-1}$
and $\phi'$ lies in $\operatorname{\normalfont A}_{g,p-1}$.

Thus, we have a natural isomorphism $\Stab(t_{p}; \operatorname{\normalfont A}_{g,p}) \xrightarrow{\null_\sim}
\operatorname{\normalfont A}_{g,p-1}$, $\phi \mapsto \phi'$.
\hfill\qed
\end{consequence}

\begin{consequence}\label{cons:z2c}
  Suppose that $p = 0$ and $g \ge 1$.

Let us identify $F_{g,0} = H {\ast} K$ where
\mbox{$H \coloneq \gp{x_1}{\quad}$} and
\mbox{$K \coloneq \gp{y_{[1{\uparrow} g]} \cup x_{[2{\uparrow} g]}}{\quad}$}.
We have an isomorphism
 \mbox{$K\xrightarrow{\null_\sim} F_{g{-}1, 1}$}  with
\mbox{$y_1 \mapsto t_1$},
and, for each $i \in [2{\uparrow}g]$, \mbox{$x_{i} \mapsto x_{i-1}$}, $y_{i} \mapsto y_{i-1}$.

Let $V \coloneq     y_1\Pi [x ,y ]_{\bbl1 2{\uparrow}g \bbr1 }$ and
 $\phi \in  \Stab(V; \operatorname{\normalfont A}_{g,0})=  \Stab(\overline x_1 \overline y_1 x_1; \operatorname{\normalfont A}_{g,0})$.
 Then  $\phi$ stabilizes the $F_{g,0}$-conjugacy class $[y_1]$.
By Theorem~\ref{thm:z}, $\phi$  acts as an automorphism on  $K$  such that the induced action on
 $F_{g{-}1,1} $  is an element
$\phi'$ of $\operatorname{\normalfont A}_{g{-}1, 1}$.

Then we have a homomorphism $\Stab(\overline x_1 \overline y_1 x_1; \operatorname{\normalfont A}_{g,0}) \to \operatorname{\normalfont A}_{g{-}1, 1}$,
 $\phi \mapsto \phi'$.
It is easily seen that this map is surjective, and that the kernel is generated by
\mbox{$\alpha_1 \coloneq  \bigl(\begin{smallmatrix}
x_1 \\   \overline y_1x_1
\end{smallmatrix}\bigr)$}.
Thus, we have an isomorphism
$\Stab(\overline x_1 \overline y_1 x_1; \operatorname{\normalfont A}_{g,0})
\simeq \gp{\alpha_1}{\quad} \times \operatorname{\normalfont A}_{g{-}1, 1}.$
\hfill\qed
\end{consequence}

\section{The canonical edges in the Zieschang groupoid}
\label{sec:mc}

In this section, we develop   methods introduced by McCool in~\cite{Mc}.
We define the canonical edges in
$\mathcal{Z}_{g,p}$
and use them to find a special generating set for    $\operatorname{\normalfont A}_{g,p}$.

Throughout this section, all products $AB$
are understood to be  without cancellation; any product  where cancellation might be possible
 will be   written  as  $A{\circ}B$.  Upper-case letters will be used to denote elements of $F_{g,p}$,
and lower-case letters will be used  to denote $t$-letters  and $x$-letters.

\begin{definitions}\label{defs:phi}   Let $g$, $p \in [0{\uparrow}\infty[$\,,
let $F_{g,p} \coloneq \gp{t_{[1 \uparrow p]} \cup x_{[1 \uparrow g ]} \cup  y_{[1 \uparrow g ]}}{\quad}$,
and let $V \in \operatorname{V}\!\!\mathcal{Z}_{g,p}$.
We shall now recursively
construct a path in  $\mathcal{Z}_{ g,p }$ from $V$ to
  \mbox{$ \Pi t_{\bbl1 p{\downarrow} 1 \bbr1 }\Pi [x,y]_{\bbl1 1{\uparrow} g \bbr1 }$.}
In particular, $\mathcal{Z}_{g,p}$ is connected.
At each step,  we specify an automorphism and  tacitly apply
Consequence~\ref{cons:z1b} to see that we have  an
edge in~$\mathcal{Z}_{g,p}$.

\medskip

\begin{enumerate}[(i).]
\item If $p \ge 1$  and   $V = Pt_jQ$  where $t_j$ is the first $t$-letter which occurs in $V$ and
$P \ne 1$, then we travel along the   edge \newline
$ Pt_jQ \xrightarrow{\left( \begin{smallmatrix}
t_j \\    t_j^{  P}
\end{smallmatrix}\right)} t_j PQ.$
\item If $p \ge 1$  and   $V =  t_jP$    and
$j \ne p$, then we travel along the  edge
\newline
$ t_jP \xrightarrow{\left( \begin{smallmatrix}
t_j && t_p \\    t_p && t_j
\end{smallmatrix}\right)}  t_pP'.$
\item If $j \in [2{\uparrow}p]$ and $V$ begins with $\Pi t_{\bbl1 p{\downarrow}(j{+}1)  \bbr1 }$
but not with
$\Pi t_{\bbl1 p{\downarrow} j \bbr1 }$, then we proceed analogously to steps (i) and (ii).
\item If $g \ge 1$ and $V = \Pi t_{\bbl1 p{\downarrow} 1 \bbr1 }aP\overline a Q$ where
$a$ is an $x$-letter and $a \ne \overline x_1$, then
we travel along the   edge \newline
 $\Pi t_{\bbl1 p{\downarrow} 1 \bbr1 } aP\overline a Q \xrightarrow{\left( \begin{smallmatrix}
a && x_1 \\    \overline x_1 && \overline a
\end{smallmatrix}\right)} \Pi t_{\bbl1 p{\downarrow} 1 \bbr1 } \overline x_1 P' x_1 Q'.$

\item Suppose that $g \ge 1$ and $V = \Pi t_{\bbl1 p{\downarrow} 1 \bbr1 }\overline x_1 P x_1 Q $
and $\vert P \vert \ge 2$.   If the set of letters
which occur in $P$ were closed under taking inverses, then the extended Whitehead graph of $V$ would have a cycle
\mbox{$ \overline {P_{\text{first}}} {\rightsquigarrow} \cdots{\rightsquigarrow}
 P_{\text{last}} {\rightsquigarrow} \overline x_1  {\rightsquigarrow} \overline {P_{\text{first}}}$},
which is a contradiction.
Let $b$ denote the first letter  that occurs in $P$  such that $\overline b$ occurs in $Q$.
We write  $P = P_1bP_2$ and $Q=Q_1 \overline b Q_2$,  and
we travel along the edge   \newline
 $\Pi t_{\bbl1 p{\downarrow} 1 \bbr1 }\overline x_1 P_1  b P_2 x_1 Q_1 \overline b Q_2
\xrightarrow{\left( \begin{smallmatrix}
b \\    \overline P_1 b \overline P_2
\end{smallmatrix}\right)}
\Pi t_{\bbl1 p{\downarrow} 1 \bbr1 } \overline x_1 b x_1 Q_1 P_2 \overline b P_1 Q_2.$
\item If $g\ge 1$  and
$V = \Pi t_{\bbl1 p{\downarrow} 1 \bbr1 }\overline x_1 b x_1 P \overline b Q$
where $b$ is an $x$-letter and $b \ne \overline y_1$, then we travel along
 the   edge \newline
 $\Pi t_{\bbl1 p{\downarrow} 1 \bbr1 }\overline x_1 b x_1 P \overline b Q
\xrightarrow{\left( \begin{smallmatrix}
b && y_1 \\    \overline y_1 && \overline b
\end{smallmatrix}\right)}
\Pi t_{\bbl1 p{\downarrow} 1 \bbr1 }\overline x_1 \overline y_1 x_1 P'  y_1 Q'.$
\item Suppose that $g\ge 1$ and  $V = \Pi t_{\bbl1 p{\downarrow} 1 \bbr1 }\overline x_1 \overline y_1 x_1 P   y_1 Q$
and $P \ne 1$.
Here the
extended Whitehead graph of  $V$ has the form\newline
$\overline t_p {\rightsquigarrow}t_p {\rightsquigarrow}{\cdot\cdot}{\rightsquigarrow}
\overline t_1{\rightsquigarrow} t_1 {\rightsquigarrow} x_1 {\rightsquigarrow}
 \overline{P_{\text{\,first}}} {\rightsquigarrow} {\cdots}{\rightsquigarrow} P_{\text{\,last}} {\rightsquigarrow} \overline y_1
{\rightsquigarrow}\overline x_1 {\rightsquigarrow} y_1 {\rightsquigarrow} \overline{Q_{\text{\,first}}}
 {\rightsquigarrow}{\cdots}{\rightsquigarrow} Q_{\text{\,last}}.$\newline
Let $\phi$ denote the (Whitehead) automorphism of $F_{g,p}$ such that, for each
letter~$u$, \newline
$u^\phi \coloneq y_1^{\operatorname{Truth}(\overline u \text{ appears in }
\overline{P_{\text{\,first}}} {\rightsquigarrow}{\cdots}{\rightsquigarrow}P_{\text{\,last}})}
\circ u\circ
\overline y_1^{\operatorname{Truth}( u \text{ appears in }
\overline{P_{\text{\,first}}} {\rightsquigarrow}{\cdots}{\rightsquigarrow} P_{\text{\,last}})} $ \newline
where $\operatorname{Truth}(-)$ assigns the value $1$ to true statements and the value $0$ to false statements.
Then $\phi$ stabilizes each $t$-letter  and $x_1$ and $y_1$.
For all but two edges $(v_k{\rightsquigarrow}\overline v_{k+1})$,
the right multiplier for $v_k$ equals the right multiplier for $\overline v_{k+1}$, that is,
the inverse of the left multiplier for $v_{k+1}$.
The two exceptional edges are $x_1{\rightsquigarrow}\overline{P_{\text{\,first}}}$ and $P_{\text{\,last}}{\rightsquigarrow}\overline y_1$.
It follows that   $Q^\phi = Q$  and $P^\phi = y_1 P \overline y_1$. We travel along the   edge  \newline
 $\Pi t_{\bbl1 p{\downarrow} 1 \bbr1 }\overline x_1 \overline y_1 x_1 P y_1 Q
\xrightarrow{\phi}
\Pi t_{\bbl1 p{\downarrow} 1 \bbr1 } [x_1,y_1] PQ.$ \vspace{1mm}
\item If $i \in [2{\uparrow}g]$ and $V$ begins with
$\Pi t_{\bbl1 p{\downarrow} 1 \bbr1 }\Pi [x,y]_{\bbl1 1{\uparrow} (i{-}1) \bbr1 }$ but $V$
does not begin with
$\Pi t_{\bbl1 p{\downarrow} 1 \bbr1 }\Pi [x,y]_{\bbl1 1{\uparrow} i \bbr1 }$,
then we proceed analogously to steps (iv)--(vii).
\end{enumerate}

\medskip

 The foregoing procedure specifies
a path in $\mathcal{Z}_{g,p}$ from $V$
 to $\Pi t_{\bbl1 p{\downarrow} 1 \bbr1 }\Pi [x,y]_{\bbl1 1{\uparrow} g \bbr1 }$,
and, hence, a \textit{canonical edge} in  $\mathcal{Z}_{g,p}$, denoted
$$ V \xrightarrow{\Phi_V}  \Pi t_{\bbl1 p{\downarrow} 1 \bbr1 }\Pi [x,y]_{\bbl1 1{\uparrow} g \bbr1 }.$$
We understand that $\Phi_{\Pi t_{\bbl1 p{\downarrow} 1 \bbr1 }\Pi [x,y]_{\bbl1 1{\uparrow} g \bbr1 }}$ is the identity map.
 The only information about $\Phi_V$ that we shall
need is that  the following hold; all of these assertions can be seen from the construction.

\begin{enumerate}[(\ref{defs:phi}.1)]
\item If $p=0$ and $g \ge 1$ and $V =  a P \overline a Q$, then
$\Phi_V$ sends $a$ to $\overline x_1$, and $P$ to $\overline y_1$.
\item If $p=1$  and $V = Pt_{1} Q  = (t_{1}^{\overline P})\circ(PQ)$, then
$\Phi_V$ sends $t_{1}^{\overline P}$ to $t_1$.
\item If $p=1$ and $g \ge 1$ and $V =  t_{1} a P \overline a Q$, then
$\Phi_V$ sends $t_{ 1} $ to $t_1$,
$a$ to $\overline x_1$,  and  $P$ to $\overline y_1$.
\item If $p \ge 2$ and $V = Pt_{j_1}Qt_{j_2}R =
 (t_{j_1}^{\overline P}) {\circ} (
 t_{j_2}^{\overline Q\, \overline P}) {\circ}  (PQR)$,
 and no $t$-letters occur in~$P$ ~or~$Q$, then
$\Phi_V$ sends $t_{j_1}^{\overline P}$ to $t_p$, and $t_{j_2}^{\overline Q\, \overline P}$ to $t_{p-1}$.
\end{enumerate}
\vskip-5mm \hfill\qed
\end{definitions}

\medskip

\begin{remark}~\label{rem}  We shall be given a special subset $A'$ of $\operatorname{\normalfont A}_{g,p}$
that we wish to show generates~$\operatorname{\normalfont A}_{g,p}$.
We view $\operatorname{\normalfont A}_{g,p}$ as the set of edges of $\mathcal{Z}_{g,p}$
from
$ \Pi t_{\bbl1 p{\downarrow} 1 \bbr1 }\Pi [x,y]_{\bbl1 1{\uparrow} g \bbr1 }$
to itself,
and we let  $\mathcal{Z}_{g,p}'$  denote the
subgroupoid of $\mathcal{Z}_{g,p}$ generated by the edges in  $A'$ together with all the
canonical edges of $\mathcal{Z}_{g,p}$.
Using methods introduced by \mbox{McCool \cite{Mc}}, we shall prove that   $\mathcal{Z}_{g,p}'$
contains the Nielsen subgraph $\mathcal{N}_{g,p}$ of $\mathcal{Z}_{g,p}$.
By Consequence~\ref{cons:z1},
$\mathcal{Z}_{g,p}' =  \mathcal{Z}_{g,p}$.  Now when any
edge in $ \operatorname{\normalfont A}_{g,p}$
is expressed as a product of canonical edges and edges in~$A'$ and their inverses,
then  the nontrivial canonical edges and their inverses must pair off
and  cancel  out,
and we are left with an expression that involves no nontrivial canonical edges.
Here, $\operatorname{\normalfont A}_{g,p}$ is generated by~$A'$.
\hfill\qed
 \end{remark}

\begin{theorem}\label{thm:p0} Let $g \in [1{\uparrow}\infty[$\,, $p=0$.
Then the group $\operatorname{\normalfont A}_{g,0}$
is generated by
$\Stab(\overline x_1 \overline y_1 x_1; \operatorname{\normalfont A}_{g,0})
\cup \{\beta_1\}$, where $\beta_1 \coloneq \bigl(\begin{smallmatrix}
y_1 \\    x_1y_1
\end{smallmatrix}\bigr)$.

\end{theorem}

\begin{proof}  Let $\mathcal{Z}_{g,0}'$  denote the
subgroupoid of $\mathcal{Z}_{g,0}$ generated by  the given set together with all the
canonical edges.    By Remark~\ref{rem}, it suffices to show that   $\mathcal{N}_{g,0} \subseteq \mathcal{Z}_{g,0}'$.

Recall that   $\alpha_1 \coloneq \bigl(\begin{smallmatrix}
x_1 \\   \overline y_1x_1
\end{smallmatrix}\bigr)   \in \Stab(\overline x_1 \overline y_1 x_1; \operatorname{\normalfont A}_{g,0}) \subseteq \mathcal{Z}_{g,0}'$.
In $F_{g,0}$, $(\overline x_1 \overline y_1 x_1)^{\beta_1\alpha_1}
= (\overline x_1 \overline y_1)^{\alpha_1} = \overline x_1$.
Hence  $\Stab(\overline x_1; \operatorname{\normalfont A}_{g,0}) \subseteq \mathcal{Z}_{g,0}'$.
Thus $\mathcal{Z}_{g,0}'$ contains all the edges of the forms

\bigskip

\noindent $(\M\text{I}.1) \colon   V \in  \operatorname{V}\!\!\mathcal{Z}_{g,0}
\xrightarrow{ \Phi_V}
  \Pi [x,y]_{\bbl1 1{\uparrow} g \bbr1 }$,

\bigskip

\noindent
$(\M\text{I}.2)  \colon \Pi [x,y]_{\bbl1 1{\uparrow} g \bbr1 }
\xrightarrow{\text{map in $\operatorname{\normalfont A}_{g,0}$ that stabilizes $\overline x_1$ or
$\overline x_1 \overline y_1 x_1$}}
\Pi [x,y]_{\bbl1 1{\uparrow} g \bbr1 }.$

\bigskip

We next  describe two more families of edges in $\mathcal{Z}_{g,0}'$,
expressed as products of edges of types $(\M\text{I}.1)$ and $(\M\text{I}.2)$ and their inverses.

\begin{figure}[H]
\begin{center}
\setlength{\unitlength}{1cm}
\begin{picture}(12.5,1.9)
\put(0,1.7){\makebox(0,0)[l] {$(\M\text{I}.3) \colon a_1P_1  \in   \operatorname{V}\!\!\mathcal{Z}_{g,0} $}}
\put(3.2,1.7){\vector(1,0){5.4}}
\put(3.5,1.9){\makebox(0,0)[l] {$\scriptstyle \text{map in $\Aut F_{g,0}$ with }  a_1 \mapsto a_2,\,
P_1 \mapsto   P_2 $}}
\put(8.8,1.7){\makebox(0,0)[l] {$a_2 P_2  \in  \operatorname{V}\!\!\mathcal{Z}_{g,0}$}}
\put(1,1.5){\vector(0,-1){1}}
\put(1.1,1.0){\makebox(0,0)[l]{$\scriptstyle (\M\text{I}.1)  \overset{(\ref{defs:phi}.1) }{\Rightarrow}  ( a_1 \mapsto \overline x_1)    $}}
\put(8.9,1.5){\vector(0,-1){1}}
\put(9,1.0){\makebox(0,0)[l]{$\scriptstyle (\M\text{I}.1)  \overset{(\ref{defs:phi}.1) }{\Rightarrow}  (  a_2 \mapsto \overline x_1 )    $}}
\put(.9,.2){\makebox(0,0)[l]{$\Pi [x,y]_{\bbl1 1{\uparrow} g \bbr1 } $}}
\put(8.8,.2){\makebox(0,0)[l]{$\Pi [x,y]_{\bbl1 1{\uparrow} g \bbr1 } $}}
\put(2.7,.2){\vector(1,0){6}}
\put(2.8,0){\makebox(0,0)[l]{$ \scriptstyle   \text{map  making square commute }  \Rightarrow
 (\overline x_1 \mapsto  \overline x_1)  \Rightarrow \, (\M\text{I}.2)$}}
\end{picture}
\end{center}
\end{figure}

\begin{figure}[H]
\begin{center}
\setlength{\unitlength}{1cm}
\begin{picture}(12.5,1.9)
\put(0,1.7){\makebox(0,0)[l] {$(\M\text{I}.4) \colon
abP \overline a  Q   \in \operatorname{V}\!\!\mathcal{Z}_{g,0} $}}
\put(3.8,1.7){\vector(1,0){5.2}}
\put(4.7,2.0){\makebox(0,0)[l] {$\scriptstyle \left(\begin{smallmatrix}
a \\
a \overline b
\end{smallmatrix}\right) \Rightarrow (abP  \overline a\, \,\mapsto \,\, aP b \overline a)$}}
\put(9.1,1.7){\makebox(0,0)[l] {$a  P b \overline  a  Q   \in  \operatorname{V}\!\!\mathcal{Z}_{g,0}$}}
\put(1.0,1.5){\vector(0,-1){1}}
\put(1.1,1.0){\makebox(0,0)[l]{$\scriptstyle (\M\text{I}.1)   \overset{(\ref{defs:phi}.1) }{\Rightarrow}   (abP \overline a
\,\mapsto \,\overline x_1\overline y_1 x_1 )   $}}
\put(9.1,1.5){\vector(0,-1){1}}
\put(9.2,1.0){\makebox(0,0)[l]{$\scriptstyle (\M\text{I}.1)   \overset{(\ref{defs:phi}.1) }{\Rightarrow}
( a P b\overline a  \,\mapsto \,\overline x_1\overline y_1 x_1 )  $}}
\put(0.9,.2){\makebox(0,0)[l]{$\Pi [x,y]_{\bbl1 1{\uparrow} g \bbr1 } $}}
\put(9.0,.2){\makebox(0,0)[l]{$\Pi [x,y]_{\bbl1 1{\uparrow} g \bbr1 } $}}
\put(2.7,.2){\vector(1,0){6.2}}
\put(2.75,0){\makebox(0,0)[l]{$ \scriptstyle   \text{make  square commute } \Rightarrow
 ( \overline x_1\overline y_1 x_1 \mapsto  \overline x_1\overline y_1 x_1)  \Rightarrow \, (\M\text{I}.2)$}}
\end{picture}
\end{center}
\end{figure}

We then have the family

\noindent $(\M\text{I}.5) \colon
abP \overline b Q   \in \operatorname{V}\!\!\mathcal{Z}_{g,0}  \xrightarrow{\left(\begin{smallmatrix}
b \\
\overline a  b
\end{smallmatrix}\right)} b P  \overline b a Q  \in \operatorname{V}\!\!\mathcal{Z}_{g,0},$

\medskip

\noindent since, here, we have the factorization

\medskip

\noindent \mbox{$abP \overline b Q    \xrightarrow{ \left(\begin{smallmatrix}
a \\
a \overline b
\end{smallmatrix}\right) \Rightarrow (\M\text{I}.4)}aP' \overline b Q'
 \xrightarrow{ \text{ makes triangle commute } \Rightarrow (a \mapsto b)\Rightarrow (\M\text{I}.3)}
 b P  \overline b a Q$.}

\medskip

 It can be seen that the edges of type $(\M\text{I}.3)$  include all the $\N_1$ edges in $\mathcal{Z}_{g,0}$, and also all
the $\N_2$ edges in $\mathcal{Z}_{g,0}$ that do not involve $a_1$.
The remaining $\N_2$ edges in $\mathcal{Z}_{g,0}$ are of type  $(\M\text{I}.4)$ or   $(\M\text{I}.5)$ or their inverses.
 Since $p=0$, there
are no $\N_3$ edges.
We have now shown that    $\mathcal{N}_{g,0} \subseteq \mathcal{Z}_{g,0}'$, as desired.
\end{proof}

\begin{theorem}\label{thm:p1} Let $g \in [1{\uparrow}\infty[$\,, $p=1$.
Then  the group $\operatorname{\normalfont A}_{g,1}$
is generated by
$\Stab(t_1; \operatorname{\normalfont A}_{g,1})
\cup \{\gamma_1\}$ where $\gamma_1 \coloneq \bigl(\begin{smallmatrix}
t_1 & & x_1  \\  \\ t_1^{w_1} & & x_{1}w_{1}
\end{smallmatrix}\bigr)$ with \mbox{$w_1 \coloneq  t_1 \overline y_1^{x_1}$}.
\end{theorem}

\begin{proof} Let $\mathcal{Z}_{g,1}'$  denote the
subgroupoid of $\mathcal{Z}_{g,1}$ generated by  the given set together with all the
canonical edges.     By Remark~\ref{rem}, it suffices to show that   $\mathcal{N}_{g,1} \subseteq \mathcal{Z}_{g,1}'$.

Now $\mathcal{Z}_{g,1}'$ contains
$\Stab(t_1; \operatorname{\normalfont A}_{g,1})\gamma_1$,  which consists of the maps in $\operatorname{\normalfont A}_{g,1}$
with $t_1 \mapsto   t_1^{\gamma_1} = t_1^{w_1} =  t_1^{\overline x_1 \overline y_1 x_1}$.
Thus,  $\mathcal{Z}_{g,1}'$  contains all the  edges of the forms

\bigskip

\noindent $(\M\text{II}.1) \colon   V \in  \operatorname{V}\!\!\mathcal{Z}_{g,1}
\xrightarrow{ \Phi_V}
  t_1\Pi [x,y]_{\bbl1 1{\uparrow} g \bbr1 },$

\bigskip

\noindent $(\M\text{II}.2)  \colon t_1\Pi [x,y]_{\bbl1 1{\uparrow} g \bbr1 }
\xrightarrow{\text{map in $\operatorname{\normalfont A}_{g,1}$  with } t_1 \mapsto t_1 \text{ or } t_1 \mapsto t_1^{\overline x_1 \overline y_1 x_1}}
t_1\Pi [x,y]_{\bbl1 1{\uparrow} g \bbr1 }.$

\bigskip

We next   describe another family of edges in $\mathcal{Z}_{g,1}'$.

\begin{figure}[H]
\begin{center}
\setlength{\unitlength}{1cm}
\begin{picture}(12.5,1.95)
\put(0,1.7){\makebox(0,0)[l] {$(\M\text{II}.3) \colon P_1t_{1}Q_1   \in \operatorname{V}\!\!\mathcal{Z}_{g,1}$}}
\put(3.5,1.7){\vector(1,0){6.2}}
\put(3.7,1.95){\makebox(0,0)[l] {$\scriptstyle \text{ map in $ \Aut F_{g,1} $   with }
t_{1}^{\overline P_1} \mapsto t_{1}^{\overline P_2},\, P_1Q_1 \mapsto P_2Q_2 $}}
\put(9.8,1.7){\makebox(0,0)[l] {$  P_{2}t_{1}Q_{2} \in \operatorname{V}\!\!\mathcal{Z}_{g,1}$}}
\put(1.2,1.5){\vector(0,-1){1}}
\put(1.3,1.1 ){\makebox(0,0)[l]{$\scriptstyle (\M\text{II}.1)  \overset{(\ref{defs:phi}.2) }{\Rightarrow}  (t_{ 1}^{\overline P_{ 1}} \mapsto  t_1)$}}
\put(9.9,1.5){\vector(0,-1){1}}
\put(10.0,1.1){\makebox(0,0)[l]{$\scriptstyle (\M\text{II}.1)   \overset{(\ref{defs:phi}.2) }{\Rightarrow}   (t_{1}^{\overline P_{ 2}}  \mapsto t_1)$}}
\put(1.1,.2){\makebox(0,0)[l]{$t_1 \Pi  [x,y]_{\bbl1 1 {\uparrow} g \bbr1 }  $}}
\put(9.8,.2){\makebox(0,0)[l]{$t_1 \Pi  [x,y]_{\bbl1 1 {\uparrow} g \bbr1 } $}}
\put(3.2,.2){\vector(1,0){6.5}}
\put(3.3,0){\makebox(0,0)[l]{$ \scriptstyle   \text{map  makes the square commute }  \Rightarrow  (t_1 \mapsto t_1)   \Rightarrow \, (\M\text{II}.2)$}}
\end{picture}
\end{center}
\end{figure}

Edges of type $(\M \text{II}.3)$ include all the $\N_1$ edges, and all the
$\N_2$ edges which do not involve $t_{1}$, and  all the $\N_3$ edges, since these have the form
\mbox{
$Pat_{1} Q
\xrightarrow{\left ( \begin{smallmatrix}
t_{1}  \\
 t_{1}^{a}
\end{smallmatrix}\right) \Rightarrow(t_{1}^{\overline a\,\overline P} \mapsto t_{1}^{\overline P}, PaQ \mapsto PaQ)} P   t_{1}a Q     $}, or its inverse.

\bigskip

It remains to consider the $\N_2$ edges which involve $t_1$;
these are   of the forms

\noindent $Pt_{1}a Q \overline a R
\xrightarrow{\left ( \begin{smallmatrix}
a  \\
 \overline t_1 a
\end{smallmatrix}\right)} P a Q \overline a t_1 R,$ \quad
 $
Pat_{1} Q \overline a R
\xrightarrow{\left ( \begin{smallmatrix}
a   \\
 a \overline t_1
\end{smallmatrix}\right)} P a Q t_1 \overline a R      ,$   and their inverses.
To construct a commuting hexagon, we define the following edges.

\noindent $(\M\text{II}.4)\colon Pa Q \overline at_{1} R   \in \operatorname{V}\!\!\mathcal{Z}_{g,1}
 \xrightarrow{ \left ( \begin{smallmatrix}
t_1  \\
   t_1^{PaQ\overline a}
\end{smallmatrix}\right) \Rightarrow (\M\text{II}.3), (t_1^{\overline P} \mapsto  t_1^{PaQ\overline a \overline P}) }
 t_{1}Pa Q \overline a R  \in   \operatorname{V}\!\!\mathcal{Z}_{g,1},$

\medskip

\noindent $(\M\text{II}.5)\colon t_{1}Pa Q \overline a R  \in \operatorname{V}\!\!\mathcal{Z}_{g,1} \xrightarrow{\left ( \begin{smallmatrix}
a  \\
  \overline P a
\end{smallmatrix}\right)
 \Rightarrow (\M\text{II}.3), (t_1^{PaQ\overline a\overline P} \mapsto t_1^{aQ \overline a})} t_{1}a Q \overline a PR
 \in \operatorname{V}\!\!\mathcal{Z}_{g,1}$,

\medskip

\noindent $(\M\text{II}.6)\colon t_{1}a Q \overline a PR  \in\operatorname{V}\!\!\mathcal{Z}_{g,1}
 \xrightarrow{(\M\text{II}.1)
\overset{(\ref{defs:phi}.3)}{\Rightarrow} (t_1^{a Q \overline a} \mapsto t_1^{\overline x_1 \overline y_1 x_1})
} t_1
 \Pi  [x,y]_{\bbl1 1 {\uparrow} g \bbr1 }$.

\medskip

  \noindent Then we have the factorization
\begin{figure}[H]
\begin{center}
\setlength{\unitlength}{1cm}
\begin{picture}(12.5,2.0)
\put(0,1.7){\makebox(0,0)[l] {$(\M\text{II}.7) \colon Pt_1a Q \overline a R   \in \operatorname{V}\!\!\mathcal{Z}_{g,1}$}}
\put(3.8,1.7){\vector(1,0){5.9}}
\put(5.5,2.0){\makebox(0,0)[l] {$\scriptstyle  \left ( \begin{smallmatrix}
a   \\
 \overline t_1 a
\end{smallmatrix}\right)   \Rightarrow (t_1^{\overline P} \mapsto t_1^{\overline P}) $}}
\put(9.8,1.7){\makebox(0,0)[l] {$  P a Q \overline a t_1 R   \in \operatorname{V}\!\!\mathcal{Z}_{g,1}$}}
\put(1.1,1.5){\vector(0,-1){1}}
\put(1.2,1.1){\makebox(0,0)[l]{$\scriptstyle  (\M\text{II}.1)  \overset{(\ref{defs:phi}.2) }{\Rightarrow}  (t_{1}^{\overline P}  \mapsto t_1)$}}
\put(9.8,1.5){\vector(0,-1){1}}
\put(9.9,1.2){\makebox(0,0)[l]{$\scriptstyle  (\M\text{II}.4)-(\M\text{II}.6)    $}}
\put(9.9,.9){\makebox(0,0)[l]{$\scriptstyle   \Rightarrow
(t_{1}^{\overline P}  \mapsto t_1^{\overline x_1 \overline y_1 x_1})$}}
\put(1.0,0.2){\makebox(0,0)[l]{$ t_1
\Pi  [x,y]_{\bbl1 1 {\uparrow} g \bbr1 }$}}
\put(9.8,0.2){\makebox(0,0)[l]{$t_1
 \Pi  [x,y]_{\bbl1 1 {\uparrow} g \bbr1 }$}}
\put(3.1,.2){\vector(1,0){6.5}}
\put(3.2,0.0){\makebox(0,0)[l]{$ \scriptstyle   \text{map makes  square  commute } \Rightarrow    (t_1 \mapsto
t_1^{\overline x_1 \overline y_1 x_1})  \Rightarrow (\M\text{II}.2) $}}
\end{picture}
\end{center}
\end{figure}

\noindent We also have the factorization

\begin{figure}[H]
\begin{center}
\setlength{\unitlength}{1cm}
\begin{picture}(12.5,2.0)
\put(0.0,1.7){\makebox(0,0)[l] {$(\M\text{II}.8) \colon Pa t_1Q \overline a R   \in \operatorname{V}\!\!\mathcal{Z}_{g,1}$}}
\put(3.8,1.7){\vector(1,0){5.8}}
\put(6,2.0){\makebox(0,0)[l] {$\scriptstyle  \left ( \begin{smallmatrix}
a   \\
 a\overline t_1
\end{smallmatrix}\right)   $}}
\put(9.7,1.7){\makebox(0,0)[l] {$  P a Q  t_1\overline a R   \in \operatorname{V}\!\!\mathcal{Z}_{g,1}$}}
\put(1.1,1.5){\vector(0,-1){1}}
\put(1.2,1.0){\makebox(0,0)[l]{$\scriptstyle   \left ( \begin{smallmatrix}
t_1  \\
   t_1^{a}
\end{smallmatrix}\right)  \Rightarrow (\M\text{II}.3)  $}}
\put(9.7,1.5){\vector(0,-1){1}}
\put(9.8,1.0){\makebox(0,0)[l]{$\scriptstyle    \left ( \begin{smallmatrix}
t_1  \\
   t_1^{a}
\end{smallmatrix}\right)  \Rightarrow (\M\text{II}.3)   $}}
\put(1.0,.2){\makebox(0,0)[l]{$ Pt_1aQ \overline a  R \in \operatorname{V}\!\!\mathcal{Z}_{g,1}$}}
\put(9.7,.2){\makebox(0,0)[l]{$ PaQ \overline a  t_1R \in \operatorname{V}\!\!\mathcal{Z}_{g,1}$}}
\put(3.8,.3){\vector(1,0){5.8}}
\put(3.9,.0){\makebox(0,0)[l]{$ \scriptstyle   \text{map makes square commute }  \Rightarrow   \left ( \begin{smallmatrix}
a   \\
 \overline t_1  a
\end{smallmatrix}\right) \Rightarrow (\M\text{II}.7)$}}
\end{picture}
\end{center}
\end{figure}

We have now shown that    $\mathcal{N}_{g,1} \subseteq \mathcal{Z}_{g,1}'$, as desired.
\end{proof}

\begin{theorem}\label{thm:p2} Let $g \in [0{\uparrow}\infty[$\,, $p \in [2{\uparrow} \infty[$\,.
Then the group $\operatorname{\normalfont A}_{g,p}$
is generated by
$\Stab(t_p; \operatorname{\normalfont A}_{g,p})
\cup \{\sigma_p\}$  where $\sigma_p \coloneq \bigl(\begin{smallmatrix}
t_p & & t_{p-1} \\   t_{p-1} & & t_p^{t_{p-1}}
\end{smallmatrix}\bigr)$.
\end{theorem}

\begin{proof}   Let $\mathcal{Z}_{g,p}'$  denote the
subgroupoid of $\mathcal{Z}_{g,p}$ generated by  the given set together with all the
canonical edges.     By Remark~\ref{rem}, it suffices to show that   $\mathcal{N}_{g,p} \subseteq \mathcal{Z}_{g,p}'$.

Now  $\mathcal{Z}_{g,p}'$  contains
$\Stab(t_p; \operatorname{\normalfont A}_{g,p})\sigma_p$,  which consists of the maps in $\operatorname{\normalfont A}_{g,p}$
with $t_p \mapsto   t_p^{\sigma_p} = t_{p-1}$.
Thus $\mathcal{Z}_{g,p}'$  contains all the  edges of the forms

\bigskip

\noindent $(\M\text{III}.1) \colon   V \in  \operatorname{V}\!\!\mathcal{Z}_{g,p}
\xrightarrow{ \Phi_V}
   \Pi t_{\bbl1 p{\downarrow} 1 \bbr1 }\Pi [x,y]_{\bbl1 1{\uparrow} g \bbr1 },$

\bigskip

\noindent $(\M\text{III}.2)  \colon \Pi t_{\bbl1 p{\downarrow} 1 \bbr1 }\Pi [x,y]_{\bbl1 1{\uparrow} g \bbr1 }
\xrightarrow{\text{map in $ \operatorname{\normalfont A}_{g,p}$ with } t_p \mapsto t_p \text{ or } t_p \mapsto t_{p-1}}
 \Pi t_{\bbl1 p{\downarrow} 1 \bbr1 }\Pi [x,y]_{\bbl1 1{\uparrow} g \bbr1 }.$

\bigskip

We now  describe some more families of edges in $\mathcal{Z}_{g,p}'$.

In the following, we assume that no $t$-letters occur in $P_1$ or $P_2$.
\begin{figure}[H]
\begin{center}
\setlength{\unitlength}{1cm}
\begin{picture}(12.5,1.95)
\put(0,1.7){\makebox(0,0)[l] {$(\M\text{III}.3) \colon P_1t_{j_1}Q_1   \in \operatorname{V}\!\!\mathcal{Z}_{g,p}$}}
\put(3.8,1.7){\vector(1,0){5.7}}
\put(3.8,1.95){\makebox(0,0)[l] {$\scriptstyle \text{in $\Stab([t]_{[1{\uparrow}p]};F_{g,p})$, }
 P_1  \mapsto P_2,  t_{j_1} \mapsto t_{j_2} ,\,  Q_1 \mapsto  Q_2$}}
\put(9.7,1.7){\makebox(0,0)[l] {$  P_{2}t_{j_2}Q_{2} \in \operatorname{V}\!\!\mathcal{Z}_{g,p}$}}
\put(1.3,1.5){\vector(0,-1){1}}
\put(1.4,1.0){\makebox(0,0)[l]{$\scriptstyle (\M\text{III}.1)   \overset{(\ref{defs:phi}.4) }{\Rightarrow}  (t_{j_1}^{\overline P_{ 1}}  \mapsto t_p)$}}
\put(9.8,1.5){\vector(0,-1){1}}
\put(9.9,1.0){\makebox(0,0)[l]{$\scriptstyle (\M\text{III}.1)  \overset{(\ref{defs:phi}.4) }{\Rightarrow}   (t_{j_2}^{\overline P_{ 2}}  \mapsto t_p)$}}
\put(1.2,.2){\makebox(0,0)[l]{$\Pi  t_{\bbl1 p {\downarrow} 1 \bbr1 } \Pi  [x,y]_{\bbl1 1 {\uparrow} g \bbr1 }$}}
\put(9.7,.2){\makebox(0,0)[l]{$\Pi  t_{\bbl1 p {\downarrow} 1 \bbr1 } \Pi  [x,y]_{\bbl1 1 {\uparrow} g \bbr1 } $}}
\put(4,.2){\vector(1,0){5.6}}
\put(4.1,.0){\makebox(0,0)[l]{$ \scriptstyle   \text{  makes  square commute }  \Rightarrow  (t_p\mapsto t_p)   \Rightarrow \, (\M\text{III}.2)$}}
\end{picture}
\end{center}
\end{figure}
\noindent  The edges of type $(\M \text{III}.3)$ include all the $\N_1$ edges.

In the following, we assume that no $t$-letters occur in  $P$ or $Q$.
\begin{figure}[H]
\begin{center}
\setlength{\unitlength}{1cm}
\begin{picture}(12.5,2.1)
\put(0,1.7){\makebox(0,0)[l] {$(\M\text{III}.4) \colon PQt_{j_1} t_{j_2}R   \in \operatorname{V}\!\!\mathcal{Z}_{g,p}$}}
\put(4.2,1.7){\vector(1,0){5.0}}
\put(5.2,2.1){\makebox(0,0)[l] {$\scriptstyle \left ( \begin{smallmatrix}
t_{j_2}  \\
 t_{j_2}^{Q   t_{j_1}  }
\end{smallmatrix}\right) \Rightarrow (t_{j_1}^{\overline Q\,\overline P} \mapsto t_{j_1}^{\overline Q\,\overline P})$}}
\put(9.3,1.7){\makebox(0,0)[l] {$  Pt_{j_2}Qt_{j_1} R \in\operatorname{V}\!\!\mathcal{Z}_{g,p}$}}
\put(1.3,1.5){\vector(0,-1){1}}
\put(1.4,1.0){\makebox(0,0)[l]{$\scriptstyle (\M\text{III}.1)  \overset{(\ref{defs:phi}.4) }{\Rightarrow}  (t_{j_1}^{\overline Q\,\overline P} \mapsto t_p)$}}
\put(9.4,1.5){\vector(0,-1){1}}
\put(9.5,1.0){\makebox(0,0)[l]{$\scriptstyle (\M\text{III}.1)  \overset{(\ref{defs:phi}.4) }{\Rightarrow}    (t_{j_1}^{\overline Q\,\overline P} \mapsto t_{p-1})$}}
\put(1.2,.2){\makebox(0,0)[l]{$\Pi  t_{\bbl1 p {\downarrow} 1 \bbr1 } \Pi  [x,y]_{\bbl1 1 {\uparrow} g \bbr1 } $}}
\put(9.3,.2){\makebox(0,0)[l]{$\Pi  t_{\bbl1 p {\downarrow} 1 \bbr1 } \Pi  [x,y]_{\bbl1 1 {\uparrow} g \bbr1 }$}}
\put(4,.2){\vector(1,0){5.2}}
\put(4.1,.0){\makebox(0,0)[l]{$ \scriptstyle   \text{ square commutes }  \Rightarrow   (t_p \mapsto t_{p-1})  \Rightarrow \, (\M\text{III}.2)$}}
\end{picture}
\end{center}
\end{figure}

In the following, we assume that no $t$-letters occur in  $P$.
\medskip \newline $(\M\text{III}.5') \colon Pt_{j_1}Qt_{j_2}R   \in \operatorname{V}\!\!\mathcal{Z}_{g,p}
\xrightarrow{\left ( \begin{smallmatrix}
t_{j_2}  \\
 t_{j_2}^{P  t_{j_1} Q}
\end{smallmatrix}\right)}   t_{j_2}Pt_{j_1}QR \in \operatorname{V}\!\!\mathcal{Z}_{g,p}$
\medskip \newline
has the factorization
\medskip \newline \mbox{$ Pt_{j_1}Qt_{j_2}R
\xrightarrow{\left ( \begin{smallmatrix}
t_{j_2}  \\
 t_{j_2}^{ Q }
\end{smallmatrix}\right)\Rightarrow (\M\text{III}.3)}   Pt_{j_1}t_{j_2}QR
\xrightarrow{\left ( \begin{smallmatrix}
t_{j_2}  \\
 t_{j_2}^{ P t_{j_1} }
\end{smallmatrix}\right) \newline \Rightarrow (\M\text{III}.4)}   t_{j_2}Pt_{j_1}QR    $.}

\medskip

In the following, we do allow   $t$-letters to occur in $P$, and rewrite $(\M\text{III}.5')$ as

  \noindent  $(\M\text{III}.5) \colon Pt_jQ  \in \operatorname{V}\!\!\mathcal{Z}_{g,p} \xrightarrow{\left ( \begin{smallmatrix}
t_j   \\
  t_j^{P}
\end{smallmatrix}\right)} t_jPQ \in \operatorname{V}\!\!\mathcal{Z}_{g,p}.$

\medskip

In the following, we do allow   $t$-letters to occur in $P_1$, $P_2$.

\begin{figure}[H]
\begin{center}
\setlength{\unitlength}{1cm}
\begin{picture}(12.5,1.9)
\put(0,1.7){\makebox(0,0)[l] {$(\M\text{III}.6) \colon P_1t_{j }Q_1   \in \operatorname{V}\!\!\mathcal{Z}_{g,p}$}}
\put(3.8,1.7){\vector(1,0){6.2}}
\put(3.8,1.9){\makebox(0,0)[l] {$\scriptstyle\text{in $\Stab([t]_{[1{\uparrow}p]};F_{g,p})$ with }   P_1\mapsto P_2, t_j\mapsto t_j, Q_1\mapsto Q_2$}}
\put(10.2,1.7){\makebox(0,0)[l] {$  P_2t_{j }Q_2 \in \operatorname{V}\!\!\mathcal{Z}_{g,p}$}}
\put(1.3,1.5){\vector(0,-1){1}}
\put(1.4,1.0){\makebox(0,0)[l]{$\scriptstyle \left ( \begin{smallmatrix}
t_{j }   \\
  t_{j }^{P_1}
\end{smallmatrix}\right)    \Rightarrow (\M\text{III}.5) $}}
\put(10.2,1.5){\vector(0,-1){1}}
\put(10.3,1.0){\makebox(0,0)[l]{$\scriptstyle \left ( \begin{smallmatrix}
t_{j }   \\
  t_{j }^{P_2}
\end{smallmatrix}\right)    \Rightarrow (\M\text{III}.5) $}}
\put(1.2,.2){\makebox(0,0)[l]{$t_{j } P_1Q_1 \in \operatorname{V}\!\!\mathcal{Z}_{g,p}$}}
\put(10.2,.2){\makebox(0,0)[l]{$ t_{j }P_2Q_2\in \operatorname{V}\!\!\mathcal{Z}_{g,p}$}}
\put(3.7,.2){\vector(1,0){6.3}}
\put(3.9,0){\makebox(0,0)[l]{$ \scriptstyle   \text{map  makes  square commute } \Rightarrow   (t_j \mapsto t_{j})  \Rightarrow \, (\M\text{III}.3)$}}
\end{picture}
\end{center}
\end{figure}

 Since $p \ge 2$,  any $\N_2$ edge of $\mathcal{Z}_{g,p}$  will be of type $(\M\text{III}.6)$ for some $j$,
as will any $\N_3$ edge except where $p=2$ and we have an edge of the    form\linebreak
  $Pt_{j_1}t_{j_2} Q   \in  \operatorname{V}\!\!\mathcal{Z}_{g,p}
\xrightarrow{\left ( \begin{smallmatrix}
t_{j_2}  \\
 t_{j_2}^{ t_{j_1}}
\end{smallmatrix}\right) } P t_{j_2} t_{j_1} Q    \in  \operatorname{V}\!\!\mathcal{Z}_{g,p},$
and or its inverse,  and, since $p=2$, these are of type $(\M\text{III}.4)$.

We have now shown that    $\mathcal{N}_{g,p} \subseteq \mathcal{Z}_{g,p}'$, as desired.  \end{proof}

\section{The ADLH generating set}\label{sec:main}

The results of the preceding two sections combine to give
 an algebraic proof of the algebraic form of~\cite[Proposition 2.10(ii) with $r=0$]{LP}.
We start with the ADL set.

\begin{theorem}\label{thm:main}  Let  $g$, $p \in [0{\uparrow}\infty[$\,.
 Let $\operatorname{\normalfont A}_{g,p}$ denote the group of automorphisms of $\gp{t_{[1 \uparrow  p ]}
 \cup x_{[1 \uparrow g ]} \cup  y_{[1 \uparrow g ]}}
{\quad}$
that fix  $\operatornamewithlimits{\Pi}  t_{\bbl1 p {\downarrow} 1 \bbr1 }
 \operatornamewithlimits{\Pi}  [x ,y ]_{\bbl1 1 {\uparrow} g \bbr1 } $ and
permute the set of  conjugacy classes \mbox{$[ t]_{[1 {\uparrow} p]} $}.
Then
$\operatorname{\normalfont A}_{g,p}$ is generated by
$$
\sigma_{[2{\uparrow}p]} \cup \alpha_{[1{\uparrow}g]} \cup \beta_{[1{\uparrow}g]}
\cup  \gamma_{[\max(2-p,1){\uparrow}g]},
$$
where,
 for \mbox{$j \in [2{\uparrow}p]$,}    \mbox{$\sigma_j \coloneq \bigl(\begin{smallmatrix}
t_j  & t_{j-1} \\    t_{j-1}  &  \,\,t_{j}^{{t_{j-1}}}
\end{smallmatrix}\bigr),$}
for $i \in [1{\uparrow}g]$,
 \mbox{$\alpha_i \coloneq  \bigl(\begin{smallmatrix}
x_i \\  \overline y_ix_i
\end{smallmatrix}\bigr)  $} and \mbox{$\beta_i \coloneq \bigl(\begin{smallmatrix}
y_i \\[.5mm]     x_iy_i
\end{smallmatrix}\bigr) $,}
for $i \in [2{\uparrow}g]$,
\mbox{ $\gamma_i \coloneq \bigl(\begin{smallmatrix}
x_{i-1} & & y_{i-1} & & x_{i}\\  \overline w_i x_{i-1} & & y_{i-1}^{w_{i}} & & x_{i}w_{i}
\end{smallmatrix}\bigr)$} with \mbox{$w_i \coloneq y_{i-1}   \overline y_{i }^{x_i}$,}
and if  $\min(1,g,p) = 1$,
 $\gamma_1 \coloneq \bigl(\begin{smallmatrix}
t_1 & & x_1  \\   t_1^{w_1} & & x_{1}w_{1}
\end{smallmatrix}\bigr)$ with \mbox{$w_1 \coloneq  t_1 \overline y_1^{x_1}$}.
\end{theorem}

\begin{proof}
 We use induction on $2g+p$.
If $2g+p \le 1$, then $\operatorname{\normalfont A}_{g,p}$ is trivial and the proposed generating set is empty.
Thus we may assume that $2g+p \ge 2$, and that the conclusion holds for smaller pairs $(g,p)$.

\medskip

\noindent \textbf{Case 1.} $p=0$.

Here $g \ge 1$.

By Consequence~\ref{cons:z2c},
we have a
homomorphism $\Stab(\overline x_1 \overline y_1 x_1; \operatorname{\normalfont A}_{g,0}) \to \operatorname{\normalfont A}_{g{-}1, 1}$
 such that the kernel is~$\gen{\alpha_1}$, and such that $\alpha_{[2{\uparrow}g]} \cup
\beta_{[2{\uparrow} g]} \cup
\gamma_{[2{\uparrow}g]}$ is  mapped bijectively to
$\alpha_{[1{\uparrow}(g{-}1)]} \cup
\beta_{[1{\uparrow}(g{-}1)]} \cup
\gamma_{[1{\uparrow}(g{-}1)]}$.  The latter is a generating set of $\operatorname{\normalfont A}_{g{-}1,1}$
by the induction hypothesis. It follows that
$\Stab(\overline x_1 \overline y_1 x_1; \operatorname{\normalfont A}_{g,0})$ is generated by
\mbox{$\alpha_{[1{\uparrow}g]} \cup
\beta_{[2{\uparrow} g]} \cup
\gamma_{[2{\uparrow}g]}$}.

 By Theorem~\ref{thm:p0}, $\operatorname{\normalfont A}_{g,0}$
is generated by $\Stab(\overline x_1 \overline y_1 x_1; \operatorname{\normalfont A}_{g,0})
\cup \{\beta_1\}$.

Hence $\operatorname{\normalfont A}_{g,0}$  is generated by
$\alpha_{[1{\uparrow}g]} \cup
\beta_{[1{\uparrow} g]} \cup
\gamma_{[2{\uparrow}g]}$, as desired.

\noindent \textbf{Case 2.} $p\ge 1$.

It follows from Consequence~\ref{cons:z2b} that we can identify $\Stab(t_p;\operatorname{\normalfont A}_{g,p})$ with
$\operatorname{\normalfont A}_{g,p{-}\!\!1}$ in a natural way.  By the induction hypothesis,
$\Stab(t_p;\operatorname{\normalfont A}_{g,p})$ is generated by
$\sigma_{[2{\uparrow}(p{-}1)]} \cup \alpha_{[1{\uparrow}g]} \cup \beta_{[1{\uparrow}g]}
\cup  \gamma_{[\max(3-p,1){\uparrow}g]}$.  We consider two cases.

\medskip

\noindent \textbf{Case 2.1.} $p=1$.

 Here $g \ge 1$.
 By Theorem~\ref{thm:p1}, $\operatorname{\normalfont A}_{g,1}$
is generated by $\Stab(t_1; \operatorname{\normalfont A}_{g,1})
\cup \{\gamma_1\}$.

Hence, $\operatorname{\normalfont A}_{g,1}$  is generated by
$\alpha_{[1{\uparrow}g]} \cup
\beta_{[1{\uparrow} g]} \cup
\gamma_{[1{\uparrow}g]}$, as desired.

\medskip

\noindent \textbf{Case 2.2.} $p \ge 2$.

 By Theorem~\ref{thm:p2}, $\operatorname{\normalfont A}_{g,p}$
is generated by $\Stab(t_p; \operatorname{\normalfont A}_{g,p})
\cup \{\sigma_p\}$.

 Hence, $\operatorname{\normalfont A}_{g,p}$  is generated by
\mbox{$\sigma_{[2{\uparrow}p]} \cup \alpha_{[1{\uparrow}g]} \cup \beta_{[1{\uparrow}g]}
\cup \gamma_{[1{\uparrow}g]}$}, as desired.
\end{proof}

We next recall Humphries' result~\cite{Humphries} that  the $\alpha_{[3{\uparrow}g]}$ part is not needed,
and, hence,  the ADHL set suffices.

\begin{corollary}\label{cor:humpf}
$\operatorname{\normalfont A}_{g,p}$ is generated by
$
\sigma_{[2{\uparrow}p]} \cup \alpha_{[1{\uparrow}\min(2,g)]} \cup \beta_{[1{\uparrow}g]}
\cup  \gamma_{[\max(2-p,1){\uparrow}g]}.$
\end{corollary}

\begin{proof}
It is not difficult to check that there exists an element of  $\operatorname{\normalfont A}_{3,0}$ given by $$\eta \coloneq
\bigl( \begin{smallmatrix}
x_1 && y_1&&x_2&& y_2&& x_3&&y_3\\
\overline y_3  \overline x_3 \overline y_3  &&    \overline y_3^{x_3 y_3}
 && x_2^{ [x_3,y_3]}&&y_2^{ [x_3,y_3]}&& x_1^{[x_2,y_2][x_3,y_3]}&& y_1^{[x_2,y_2][x_3,y_3]}
\end{smallmatrix}\bigr),$$
and that
$(\overline x_1 \overline y_1 x_1)^\eta =   y_3  $, and that
 both $ \alpha_1 \eta$ and $ \eta \alpha_3$ equal
$$ \bigl( \begin{smallmatrix}
x_1 && y_1&&x_2&& y_2&& x_3&&y_3\\
\overline y_3  \overline x_3   &&    \overline y_3^{x_3 y_3}
 && x_2^{ [x_3,y_3]}&&y_2^{ [x_3,y_3]}&& x_1^{[x_2,y_2][x_3,y_3]}&& y_1^{[x_2,y_2][x_3,y_3]}
\end{smallmatrix}\bigr).$$
By Consequence~\ref{cons:z2c}, each element of
$\Stab(\overline x_1 \overline y_1 x_1;\operatorname{\normalfont A}_{3,0})$ centralizes~$\alpha_1$,
and, hence,   each element of $\Stab(\overline x_1 \overline y_1 x_1;\operatorname{\normalfont A}_{3,0})\eta$
conjugates $\alpha_1$ into $\alpha_3$.
Notice that
  $\Stab(\overline x_1 \overline y_1 x_1;\operatorname{\normalfont A}_{3,0})\eta$
is the set of elements of $\operatorname{\normalfont A}_{3,0}$ with
\mbox{$\overline x_1 \overline y_1 x_1 \mapsto y_3$}.

One can  compute
\begin{align*}
& \overline x_1 \overline y_1  x_1
 &&\overset{\null^{\displaystyle\beta_1}}{\mapsto} &&\overline x_1 \overline y_1
&&\overset{\null^{\displaystyle\gamma_2}}{\mapsto} &&\overline x_1 \overline x_2 \overline y_2 x_2
&&\overset{\null^{\displaystyle\beta_2}}{\mapsto} &&\overline x_1 \overline x_2 \overline y_2
&&\overset{\null^{\displaystyle\alpha_2}}{\mapsto}&&\overline x_1 \overline x_2 =
 \\  &\overline x_1 \overline x_2 &&\overset{\null^{\displaystyle\gamma_3}}{\mapsto} && \overline x_1 \overline x_2 y_2 \overline x_3 \overline y_3 x_3
&&\overset{\null^{\displaystyle\beta_3}}{\mapsto}&& \overline x_1 \overline x_2 y_2 \overline x_3 \overline y_3
&&\overset{\null^{\displaystyle\beta_2}}{\mapsto}&&\overline x_1   y_2 \overline x_3 \overline y_3
&&\overset{\null^{\displaystyle\gamma_3}}{\mapsto}&&  \overline x_1  \overline x_3 =
\\  &\overline x_1 \overline x_3 &&\overset{\null^{\displaystyle\gamma_2}}{\mapsto}&&    \overline x_1 y_1  \overline x_2\overline y_2  x_2 \overline x_3
&&\overset{\null^{\displaystyle\beta_2}}{\mapsto}&&\overline x_1 y_1  \overline x_2\overline y_2  \overline x_3
&&\overset{\null^{\displaystyle\beta_1}}{\mapsto} &&   y_1  \overline x_2 \overline y_2 \overline x_3
&& \overset{\null^{\displaystyle\gamma_2}}{\mapsto}&& \overline x_2   \overline x_3 =
\\  &\overline x_2 \overline x_3 &&\overset{\null^{\displaystyle\alpha_2}}{\mapsto}&&  \overline x_2y_2  \overline x_3
&&\overset{\null^{\displaystyle\beta_2}}{\mapsto} &&y_2 \overline x_3
&& \overset{\null^{\displaystyle\gamma_3}}{\mapsto} && \overline x_3 y_3
&&\overset{\null^{\displaystyle\beta_3}}{\mapsto} &&y_3;
\end{align*}
this is the algebraic translation of~\cite[Figure~2]{Humphries}\vspace{.5mm}.
We then see that, as in~\cite{Humphries},
$ \textstyle \alpha_1^{\textstyle \beta_1 \gamma_2\beta_2 \alpha_2 \gamma_3 \beta_3  \beta_2 \gamma_3
 \gamma_2 \beta_2 \beta_1 \gamma_2 \alpha_2 \beta_2 \gamma_3 \beta_3} = \alpha_3.$
 By shifting the indices upward, we  see that $\alpha_{[3{\uparrow}g]}$ can be
removed from the ADL set and still leave a  generating set.
\end{proof}

We have now completed our objective.  For completeness, we conclude the article with
an elementary review of some  classic results.

\section{Some background on mapping-class groups}\label{sec:mcg}

\begin{notation}
Let us define $F_{g,p-1} \coloneq  \gp{t_{[1 \uparrow p]}
 \cup x_{[1 \uparrow g ]} \cup  y_{[1 \uparrow g ]}}
{\Pi t_{\bbl1 p{\downarrow} 1 \bbr1 }\Pi [x,y]_{\bbl1 1{\uparrow} g \bbr1 }}$.

Then   $F_{g,p} =\gp{t_{[1 \uparrow (p+1)]}
 \cup x_{[1 \uparrow g ]} \cup  y_{[1 \uparrow g ]}}
{\Pi t_{\bbl1 (p+1){\downarrow} 1 \bbr1 }\Pi [x,y]_{\bbl1 1{\uparrow} g \bbr1 }}$, and we still have
$F_{g,p}  =
 \gp{t_{[1 \uparrow p]}
 \cup x_{[1 \uparrow g ]} \cup  y_{[1 \uparrow g ]}}
{\quad}$ and here $\overline t_{p+1} = \Pi t_{\bbl1 p{\downarrow} 1 \bbr1 }\Pi [x,y]_{\bbl1 1{\uparrow} g \bbr1 }$.
\hfill\qed
\end{notation}

\begin{definitions}\label{defs:bound1}
We construct an
 orientable  surface $\mathbf{S}_{g,1,p}$, of genus $g$ with $p$ punctures
and one boundary component, as follows.
We start with  a vertex which will be the basepoint.   We attach
a set of $2g{+}p{+}1$ oriented edges
  $t_{[1 \uparrow (p{+}1)]}  \cup x_{[1 \uparrow g ]} \cup  y_{[1 \uparrow g ]}$.
We attach a $(4g{+}p{+}1)$-gon  with  counter-clockwise  boundary label
  $\textstyle   \Pi t_{\bbl1 (p+1){\downarrow} 1 \bbr1 }\Pi [x,y]_{\bbl1 1{\uparrow} g \bbr1 } \in
 \gp{t_{[1 \uparrow (p+1)]} \cup x_{[1 \uparrow g ]} \cup  y_{[1 \uparrow g ]}}
{\quad}$.
For each $j \in [1{\uparrow}p]$,  we attach a punctured
disk  with counterclockwise boundary label $\overline t_j$. This completes the definition
 of~$\mathbf{S}_{g,1,p}$.  Notice that the boundary of~$\mathbf{S}_{g,1,p}$ is the edge labelled~$t_{p+1}$.

We may identify $\pi_1(\mathbf{S}_{g,1,p}) = F_{g,p}$.
We call $\Stab( [\overline t]_{[1{\uparrow}p]}
\cup \{\overline t_{p+1}\} ; \Aut F_{g,p})$ the  \textit{algebraic}   \textit{mapping-class group of
$\mathbf{S}_{g,1,p}$}.  This is  our group $\operatorname{\normalfont A}_{g,p}$.  (In~\cite{DF3},
 $\operatorname{\normalfont A}_{g,p}$ is denoted  $\Aut^+_{g,0,p\perp \hat 1}$, and,
in~\cite[Proposition~7.1(v)]{DF3}, the latter group is
shown to be isomorphic to what  is there called
the orientation-preserving algebraic mapping-class group of $\mathbf{S}_{g,1,p}$,
denoted $\Out^+_{g,1,p}$.)

Let $\Aut \mathbf{S}_{g,1,p} $ denote the group of self-homeomorphisms of
 $\mathbf{S}_{g,1,p}$ which stabilize each  point on the boundary.
The quotient of $\Aut \mathbf{S}_{g,1,p}$ modulo the group of elements of $\Aut  \mathbf{S}_{g,1,p}$
which are isotopic to the identity map through a boundary-fixing isotopy
 is
called the (topological)  \textit{mapping-class group of
$\mathbf{S}_{g,1,p}$}, denoted $\operatorname{\normalfont M}^{\text{top}}_{g,1,p}$.

Then $\Aut \mathbf{S}_{g,1,p} $  acts on  $F_{g,p}$ stabilizing
$[\overline t]_{[1{\uparrow}p]}
\cup \{\overline t_{p+1}\}$, and we have
  a homomorphism
$
\operatorname{\normalfont M}^{\text{top}}_{g,1,p} \to \operatorname{\normalfont A}_{g,p}.$
\hfill\qed
\end{definitions}

\begin{definitions}\label{defs:bound0}  Let $\mathbf{S}_{g,0,p}$  denote
 the
quotient space obtained from   $\mathbf{S}_{g,1,p}$
 by   collapsing  the boundary to a point.
Then $\mathbf{S}_{g,0,p}$ is an
 orientable  surface of genus $g$ with $p$ punctures.

We may identify $\pi_1(\mathbf{S}_{g,0,p}) = F_{g,p-1}$.
We define the \textit{algebraic mapping-class group of $\mathbf{S}_{g,0,p}$} as
 $\operatorname{\normalfont M}^{\text{alg}}_{g,0,p} \coloneq \Stab([t]_{[1{\uparrow}p]}
\cup [\overline t]_{[1{\uparrow}p]} ;\Out F_{g,p-1}).$
 (In~\cite{DF3}, if $(g,p) \ne (0,0)$, $(0,1)$,  then
 $\operatorname{\normalfont M}^{\text{alg}}_{g,0,p}$ is  denoted $\Out_{g,0,p}$.)

Let $\Aut \mathbf{S}_{g,0,p} $ denote the group of self-homeomorphisms of
 $\mathbf{S}_{g,0,p}$.
The quotient of $\Aut \mathbf{S}_{g,0,p}$ modulo   the
 group of elements
which are isotopic to the identity map  is
called the (topological)  \textit{mapping-class group of
$\mathbf{S}_{g,0,p}$},  denoted $\operatorname{\normalfont M}^{\text{top}}_{g,0,p}$.

Then
  $\Aut \mathbf{S}_{g,0,p}$   acts on  $F_{g,p-1}/{\sim}$ \,
stabilizing $[t]_{[1{\uparrow}p]} \cup [\overline t]_{[1{\uparrow}p]}$.
This action factors through  a natural homomorphism
$\Aut \mathbf{S}_{g,0,p} \to \Out F_{g,p-1}$,  and we have a homomorphism
 $\operatorname{\normalfont M}^{\text{top}}_{g,0,p} \to \operatorname{\normalfont M}^{\text{alg}}_{g,0,p}$.

Consider the simply-connected case, that is,
  $F_{g,p-1} = 1$.  Then,  $(g,p)$ is either $(0,0)$ or $(0,1)$,
corresponding to the sphere  $\mathbf{S}_{0,0,0}$  and the open disk $\mathbf{S}_{0,0,1}$.
 Here,  $\operatorname{\normalfont M}^{\text{alg}}_{g,0,p}$  is trivial, while
$\operatorname{\normalfont M}^{\text{top}}_{g,0,p}$ has order two,
with one mapping class consisting of the reflections.
\hfill\qed
\end{definitions}

It has been the work of many years to show  that $\operatorname{\normalfont M}^{\text{top}}_{g,1,p} = \operatorname{\normalfont A}_{g, p}$ and to show that both are
generated by the ADLH set.  Also,
 if $(g,p) \ne (0,0),\, (0,1)$, then
\mbox{$\operatorname{\normalfont M}^{\text{top}}_{g,0,p} = \operatorname{\normalfont M}^{\text{alg}}_{g,0,p}$}, and their orientation-preserving
subgroups are generated by the   ADLH set.
The proofs developed in stages,  roughly as follows, although we are omitting many important results.

\begin{list}{$\bullet$} {}
\item In 1917, Nielsen~\cite{Nielsen0} proved that  if $(g,p) = (1,0)$ then the ADL set generates~$\operatorname{\normalfont A}_{g,p}$.
\item In 1925, Artin~\cite{A} introduced \textit{braid twists}, and proved that if $g=0$
 then  the ADL set generates $\operatorname{\normalfont A}_{g,p}$  and
 $\operatorname{\normalfont M}_{g,1,p}^\text{top}=\operatorname{\normalfont A}_{g,p}$.
\item In 1927, Nielsen~\cite{Nielsen2} presented unpublished results of Dehn and   proved  that  if  \mbox{$p=0$}
 then \mbox{$\operatorname{\normalfont M}_{g,1,p}^\text{top}$ maps onto $\operatorname{\normalfont A}_{g,p}$}, and that if $p\le 1$
then $\operatorname{\normalfont M}_{g,0,p}^\text{top}$ maps onto $\operatorname{\normalfont M}_{g,0,p}^\text{alg}$.
\item In 1928, Baer~\cite{Baer} proved that if $p=0$ then  $\operatorname{\normalfont M}_{g,0,p}^\text{top}$
embeds in $\operatorname{\normalfont M}_{g,0,p}^\text{alg}$ for all  $g \ge 2$.
\item  In 1934, Magnus~\cite{Magnus} proved that  if $g=1$  then
$\operatorname{\normalfont M}_{g,1,p}^\text{top}=\operatorname{\normalfont A}_{g,p}$ and
\mbox{$\operatorname{\normalfont M}_{g,0,p}^\text{top} = \operatorname{\normalfont M}_{g,0,p}^\text{alg}$}.
\item In 1939, Dehn~\cite{Dehn1} introduced what are now called \textit{Dehn twists}, and
proved, among other results,
that a finite number of Dehn twists generate the orientation-preserving subgroup of $\operatorname{\normalfont M}_{g,0,p}^\text{top}$.;
see~\cite[Section 10.3.c]{Dehn1}.
\item In 1964, Lickorish~\cite{Lickorish}  rediscovered and refined Dehn's 1939 methods and proved that if $p=0$ then
the ADL set generates the orientation-preserving subgroup of  $\operatorname{\normalfont M}_{g,0,p}^\text{top}$.
\item In 1966, Epstein~\cite{Epstein} refined Baer's 1928 methods and proved that
$\operatorname{\normalfont M}_{g,1,p}^\text{top}$ embeds in $\operatorname{\normalfont A}_{g,p}$ and that, if $(g,p) \ne (0,0), \, (0,1)$,
then $\operatorname{\normalfont M}_{g,0,p}^\text{top}$ embeds in $\operatorname{\normalfont M}_{g,0,p}^\text{alg}$.
\item In 1966, Zieschang~\cite[Satz~4]{Z66a},~\cite[Theorem~5.7.1]{ZVC} proved that
\mbox{$\operatorname{\normalfont M}_{g,1,p}^\text{top} \hskip-2.4pt =  \hskip-2.4pt\operatorname{\normalfont A}_{g,p}$}
and that, for \mbox{$(g,p) \ne (0,0), \, (0,1)$},
$\operatorname{\normalfont M}_{g,0,p}^\text{top} = \operatorname{\normalfont M}_{g,0,p}^\text{alg}$, and called
these results
the Baer-Dehn-Nielsen Theorem.
\item In 1979, Humphries\cite{Humphries} showed that  the ADHL set  generates the same group as the ADL set.
\item In 2001, Labru\`ere and Paris~\cite[Proposition 2.10(ii) with $r=0$]{LP} used
some of the foregoing results and a theorem of Birman~\cite{Birman}
to prove that the ADLH set generates $\operatorname{\normalfont M}_{g,1,p}^\text{top}$.
\end{list}

\section{The topological source of the ADL set}\label{sec:boundary}

In this section,
we shall recall the definitions of   Dehn twists and braid twists
and  see that the ADL set lies in
$\operatorname{\normalfont M}^{\text{top}}_{g,1,p}$.  The diagram~\cite[Figure 12]{LP}
 illustrates the elements of the  ADLH set acting on~$\mathbf{S}_{g,1,p}$.

\begin{definitions}\label{defs:ann}
Let $\mathbf{A}  \coloneq [0,1] \times  (\reals/\integers)$,  a  closed annulus.
Let $z$ denote the oriented   boundary component
$ \{1 \}  \times (\reals/\integers)$ with
basepoint $(1, \integers)$.  Let  $z'$ denote the oriented  boundary component
$ \{0 \}  \times (\reals/\integers)$ with
basepoint $(0,\integers)$.  Let $e$ denote the edge  $ [0,1]\times \{\integers\}$
oriented from $(1,\integers )$ to $(0,\integers)$.

The \textit{model Dehn twist} is the  self-homeomorphism
 $\tau$  of $\mathbf{A}  =  [0,1] \times  (\reals/\integers) $   given by
\mbox{$(x,y +\integers ) \mapsto (x,  -x +y +\integers)$.}
Notice that $\tau$ fixes every point of  $z' \cup z$,
and $\tau$ acts on $e$ as \mbox{$(x, \integers ) \mapsto (x,  1-x +\integers)$.}
Thus $e^\tau \overline e \,\overline z$ bounds a triangle; hence
 $e^\tau $ is homotopic to $ze$.

Suppose now that we have an embedding of $\mathbf{A}$ in a surface $\mathbf{S}$.
Then the image of  $z$ is an oriented simple closed curve $\mathbf{c}$,
and  $ \tau$ induces a self-homeo\-morphism of $\mathbf{S}$ which is the identity outside the copy of
$\mathbf{A}$.  We call the resulting map of $\mathbf{S}$ a (left) \textit{Dehn twist about} $\mathbf{c}$;
see~\cite{Dehn1}.
\hfill\qed
\end{definitions}

Recall the construction of $\mathbf{S}_{g,1,p}$ in Definitions~\ref{defs:bound1}.

\begin{examples}
Let $i \in[1{\uparrow}g]$.

Recall that $\overline x_i \overline y_i x_i$ is a
subword of the boundary label
of the $(4g{+}p{+}1)$-gon  used in the construction of
 $\mathbf{S}_{g,1,p}$.
We place the annulus $\mathbf{A}$ on $\mathbf{S}_{g,1,p}$
with the image of $z$ along the boundary edge labelled $\overline y_i$.
The  image of~$z'$  enters  the $(4g{+}p{+}1)$-gon  near the end of  the boundary edge labelled
 $\overline x_{i}$,
 travels   near $z = \overline y_i$,
and exits near the beginning of $x_{i}$, completing the cycle.
The only oriented edge  of the one-skeleton of  $\mathbf{S}_{g,1,p}$   that crosses $\mathbf{A}$ from
right to left  is
$x_{i}$, near its beginning.
Incident to the basepoint of $z$ are,  in clockwise order,
 the end of $ z$,   the beginning  of
  $  x_{i }$,  and  the beginning  of $z$.
The Dehn twist about $\overline y_i$   induces
$\left(\begin{smallmatrix}
x_{i}\\ \overline y_i x_i \end{smallmatrix}\right)$ on $\pi_1(\mathbf{S}_{g,1,p}) = F_{g,p}$.
   Hence $\alpha_i \in \operatorname{\normalfont M}^{\text{top}}_{g,1,p}$.

Recall that $\overline y_i x_i y_i$  is a
subword of the boundary label
of the $(4g{+}p{+}1)$-gon  used in the construction of
 $\mathbf{S}_{g,1,p}$.
We place the annulus $\mathbf{A}$ on $\mathbf{S}_{g,1,p}$
with the image of $z$ along the boundary edge labelled $x_i$.
The  image of~$z'$  enters  the $(4g{+}p{+}1)$-gon  near the end of  the boundary edge labelled
 $\overline y_{i}$,
 travels   near $z = x_i$,
and exits near the beginning of $y_{i}$, completing the cycle.
The only oriented edge  of the one-skeleton of  $\mathbf{S}_{g,1,p}$   that crosses $\mathbf{A}$ from
right to left  is
$y_{i}$, near its beginning.
Incident to the basepoint of $z$ are,  in clockwise order,
 the end of $ z$,   the beginning  of
  $  y_{i }$,  and  the beginning  of $z$.
The Dehn twist about $x_i$   induces
$\left(\begin{smallmatrix}
y_{i}\\ x_i y_i \end{smallmatrix}\right)$ on $\pi_1(\mathbf{S}_{g,1,p}) = F_{g,p}$.
    Hence $\beta_i \in \operatorname{\normalfont M}^{\text{top}}_{g,1,p}$.
\hfill\qed
\end{examples}

\begin{example}
Let $i \in[2{\uparrow}g]$.  Recall that $\overline x_{i-1} \overline y_{i-1}
x_{i-1} y_{i-1}\overline x_i \overline y_i x_i$  is a
subword of the boundary label
of the $(4g{+}p{+}1)$-gon  used in the construction of
 $\mathbf{S}_{g,1,p}$.
We place the annulus $\mathbf{A}$ on $\mathbf{S}_{g,1,p}$  with the image of $z$  marking out,
in the $(4g{+}p{+}1)$-gon,  a pentagon
with boundary label $y_{i-1}  \overline x_i  \overline y_i  x_i  z$.
The  image of~$z'$
\begin{list}{$\bullet$}{}
\item  enters (the $(4g{+}p{+}1)$-gon) near the end of (the boundary edge labelled) $\overline y_{i-1}$,
 travels counter-clockwise  near the basepoint,
  exits near the beginning of $x_{i-1}$,
\item  enters near the end of $\overline x_{i-1}$,
  travels counter-clockwise    near the basepoint,
 exits near the beginning of $\overline y_{i-1}$,
\item  enters near the end of $  y_{i-1}$,
  travels counter-clockwise near the basepoint,
 exits  near the beginning  of $\overline x_{i}$,
\item enters  near the end of  $x_{i}$,
 travels   near $z$, passing $\overline y_i$, $\overline x_i$,
 exits near the beginning of $y_{i-1}$,
\end{list}
completing the cycle.
The entrances  correspond to
 $\overline y_{i-1}{\rightsquigarrow}\overline x_{i-1}
{\rightsquigarrow} y_{i-1}{\rightsquigarrow} x_{i}$  in the extended Whitehead graph.
The oriented edges of the one-skeleton of  $\mathbf{S}_{g,1,p}$   that cross $\mathbf{A}$ from
right to left are the exits:
$x_{i-1}$ near its beginning,
 $ \overline y_{i-1}$ near its beginning,  $\overline x_{i}$ near its beginning, and
$  y_{i-1}$ near its beginning.
Incident to the basepoint of $z$ are,  in clockwise order,
 the end of $ z$,  and the beginnings of
  $  y_{i-1}$, $x_{i-1}$,  $ \overline y_{i-1}$, $\overline x_{i}$, and $z$.
Let $w_i  \coloneq  y_{i-1} \overline x_i \overline y_i x_i$.
The  Dehn twist about $\overline w_i$  induces  \mbox{ $ \bigl(\begin{smallmatrix}
 x_{i-1} && y_{i-1} && x_{i}\\  \overline w_i x_{i-1} && y_{i-1}^{w_i} && x_{i}w_{i}
\end{smallmatrix}\bigr)$} on $\pi_1(\mathbf{S}_{g,1,p}) = F_{g,p}$.
 Hence $\gamma_i \in \operatorname{\normalfont M}^{\text{top}}_{g,1,p}$.
\hfill\qed
\end{example}

\begin{example}
Suppose that $\min(g,p,1) = 1$.
 Recall that $
t_{1}\overline x_1 \overline y_1 x_1$  is a
subword of the boundary label
of the $(4g{+}p{+}1)$-gon  used in the construction of
 $\mathbf{S}_{g,1,p}$ and that $\overline t_1$ is the boundary label of the $\overline t_1$-disk.
We place the annulus $\mathbf{A}$ on $\mathbf{S}_{g,1,p}$  with  $z$   marking out,
in the $(4g{+}p{+}1)$-gon,  a pentagon
with boundary label $t_{ 1}  \overline x_1  \overline y_1  x_1  z$.
The image of $z'$
\begin{list}{$\bullet$}{}
\item   enters the $\overline t_1$-disk near the end of $\overline t_{1}$,
  travels counter-clockwise  near the basepoint,
 exits near the beginning of $\overline t_{1}$,
\item enters the $(4g{+}p{+}1)$-gon  near the end of $t_1$,
 travels counter-clockwise  near the basepoint,
  exits near the beginning of $\overline x_{1}$,
\item   enters the $(4g{+}p{+}1)$-gon  near the end of   $x_1$,  travels near $z$ passing
$\overline y_1$,~$\overline x_1$,
 exits near the beginning  of $t_1$,
\end{list}
completing the cycle. The entrances   correspond to
$ \overline t_{1}{\rightsquigarrow}   t_{1}
 {\rightsquigarrow} x_{1}$  in the extended Whitehead graph.
The oriented edges of the one-skeleton of  $\mathbf{S}_{g,1,p}$   that cross $\mathbf{A}$ from
right to left are the exits:
 $\overline t_1$ near its beginning,    $\overline x_{1}$ near its beginning,
and  $  t_{1}$ near its beginning.
Incident to the basepoint of $z$ are,  in clockwise order,
the end of $ z$,  and the beginnings of
  $  t_{1}$,    $ \overline t_{1}$, $\overline x_{1}$, and $z$.
Let \mbox{$w_1  \coloneq  t_{1} \overline x_1 \overline y_1 x_1$}.
The Dehn twist about $\overline w_1$  induces  \mbox{ $ \bigl(\begin{smallmatrix}
  t_{ 1} && x_{1}\\    t_{ 1}^{w_1} && x_{1}w_{1}
\end{smallmatrix}\bigr)$} on $\pi_1(\mathbf{S}_{g,1,p}) = F_{g,p}$.
 Hence \mbox{$\gamma_1 \in \operatorname{\normalfont M}^{\text{top}}_{g,1,p}$}.
\hfill\qed
\end{example}

\begin{definitions}  Recall the annulus $\mathbf{A} = [0,1] \times  (\reals/\integers)$ of Definitions~\ref{defs:ann}.
Let $\mathbf{D}$  denote the space that is obtained from  $\mathbf{A}$ by deleting
the two points
  $p_2 \coloneq ( \frac{1}{2},\integers)$ and $p_1 \coloneq ( \frac{1}{2}, \frac{1}{2} + \integers)$
and collapsing to a point the boundary component~$z' = \{0\} \times (\reals/\integers)$.
We take $p_0 \coloneq ( 1, \integers)$
as the basepoint of $\mathbf{D}$.

 Thus $\mathbf{D}$ is a closed disk with
two punctures, and the model Dehn twist
$\tau$ has an induced action on $\mathbf{D}$, called the \textit{model braid twist}.
We now determine the induced action on $\pi_1(\mathbf{D})$.

Let $z_2$ denote an infinitesimal clockwise circle  around  $p_2$,
and  let $z_1 \coloneq z_2^\tau$, an infinitesimal clockwise circle  around  $p_1$.
Then $\tau$ interchanges $z_2$ and $z_1$.
Let $e_2$ denote the oriented subedge of  $\overline e$ from $z_2$ to $p_0$
starting at a point $p_2'$ on $z_2$.
Let $e_1 \coloneq e_2^\tau $, an oriented subedge of $\overline e^\tau$ from
$z_1$ to $p_0$ starting at
$p_1' \coloneq p_2'^\tau$ on $z_1$.
 Then $\tau$ interchanges $p_1'$ and~$p_2'$, and
 acts on $e_1$ as \mbox{$(x,1-x+\integers) \mapsto (x,2 - 2x + \integers)$}.
Here, $e_1^\tau $ is an oriented edge from $p_2'$ to $p_0$ such that
$  e_1^\tau z\overline e_2  $ bounds a triangle;
hence, $e_1^\tau$ is homotopic to~$e_2 \overline z$.

We view $z_2^{ e_2}$ and $z_1^{ e_1}$  as  closed paths,
and then $z_2^{ e_2}z_1^{ e_1}z$  bounds a disk  in $\mathbf{D}$.
Now $\pi_1(\mathbf{D}) = \gp{z_2^{ e_2}, z_1^{ e_1}, z}
{z_2^{ e_2}z_1^{e_1}z} =  \gp{z_2^{ e_2}, z_1^{   e_1} }
{\quad}$, and the induced action of $\tau$ on $\pi_1(\mathbf{D})$ is given by
$z_2^{  e_2} \mapsto z_1^{  e_1}$ and
$z_1^{  e_1} \mapsto z_2^{  e_2 \, \overline z} =
z_2^{e_2 \,  z_2^{e_2} z_1^{e_1}} =
(z_2^{e_2})^{z_1^{e_1}} $.

Suppose that we have an embedding of $\mathbf{D}$
in a surface $\mathbf{S}$ which carries punctures to punctures.
Then  $ \tau$ induces a self-homeo\-morphism of
$\mathbf{S}$
which is the identity outside the copy of
$\mathbf{D}$.  The resulting map of $\mathbf{S}$ is  called a \textit{braid twist}; see~\cite{A}.
\hfill\qed
\end{definitions}

\begin{example} Let $j \in [2{\uparrow}p]$.
We   place the twice-punctured disk $\mathbf{D}$ on $\mathbf{S}_{g,1,p}$  with the image of
 $z$  marking out,
in the $(4g{+}p{+}1)$-gon,  a triangle
with boundary label $t_{j} t_{j-1}  z$.  This is possible since   $z$ now bounds a twice-punctured disk in
$\mathbf{S}_{g,1,p}$.  Here $t_j$ is homotopic to $z_2^{  e_2}$ and
$t_{j-1}$ is homotopic to $z_1^{  e_1}$.
 The resulting braid twist of $\mathbf{S}_{g,1,p}$
 induces $\bigl(\begin{smallmatrix}
t_j  && t_{j-1} \\    t_{j-1}  &&   t_{j}^{t_{j-1}}
\end{smallmatrix}\bigr)$.
 Hence $\sigma_j \in \operatorname{\normalfont M}^{\text{top}}_{g,1,p}$.
\hfill \qed
\end{example}

We now see that the ADL set lies in  $\operatorname{\normalfont M}^{\text{top}}_{g,1,p}$.
By Theorem~\ref{thm:main},
the homomorphism $\operatorname{\normalfont M}^{\text{top}}_{g,1,p} \to \operatorname{\normalfont A}_{g,p}$ is surjective;
that is, by using Zieschang's proof, we have recovered Zieschang's
result~\cite[Satz~4]{Z66a},~\cite[Theorem~5.7.1]{ZVC}.
 Assuming Epstein's result~\cite{Epstein},
we now have $\operatorname{\normalfont M}^{\text{top}}_{g,1,p} = \operatorname{\normalfont A}_{g,p}$,
and both are generated by the ADLH set.

\section{Collapsing the boundary}\label{sec:noboundary}

In this section we review Zieschang's algebraic proof of a result of
Nielsen.  We then describe a generating set for $\operatorname{\normalfont M}^{\text{alg}}_{g,0,p}$  which
 lies in the image of $\operatorname{\normalfont M}^{\text{top}}_{g,0,p}$.

\begin{definitions}\label{defs:zeta}  Recall
 $F_{g,p-1}  = \gp{t_{[1 \uparrow p]}
 \cup x_{[1 \uparrow g ]} \cup  y_{[1 \uparrow g ]}}
{\Pi t_{\bbl1 p{\downarrow} 1 \bbr1 }\Pi [x,y]_{\bbl1 1{\uparrow} g \bbr1 }}$.

Let
  \mbox{$\zeta \in \Aut F_{g,p-1}$}   be defined by
$$\forall i \in [1{\uparrow}g]  \quad\hskip-4pt x_i^\zeta \coloneq  y_{g+1-i} ,
\quad\hskip-4pt  y_i^\zeta \coloneq   x_{g+1-i} , \quad\hskip-4pt  \forall j \in [1{\uparrow}p]
 \quad \hskip-4pt  t_j^\zeta \coloneq \overline t_{p+1-j}.$$
We then have the outer automorphism $\breve \zeta \in \operatorname{\normalfont M}^{\text{alg}}_{g,0,p}$.
\hfill\qed
\end{definitions}

\begin{theorem}\label{thm:niel} For $g$, $p \in [0{\uparrow}\infty[$\,,
$\operatorname{\normalfont M}_{g,0,p}^{\text{\normalfont alg}}$ is generated by the natural image of
$\operatorname{\normalfont A}_{g,p}$ together with $\breve \zeta$.  Hence,
$\operatorname{\normalfont M}_{g,0,p}^{\text{\normalfont alg}}$ is generated
by  the image of the ADLH set together with $\breve \zeta$.
\end{theorem}

\begin{proof}[Sketched proof] For $p \ge 1$, this is a straightforward exercise which we leave to the reader.
Thus we may assume that $p=0$.  We may further assume that $g \ge 1$.
The remaining  case is now a result of Nielsen~\cite{Nielsen2} for which
Zieschang has given an algebraic proof~\cite[Theorem~5.6.1]{ZVC}
developed from~\cite{Z64D, Z65H, Z66} along the following lines.

Let  $\phi \in \Aut F_{g,-1}$. We wish to show that  the element
$\breve \phi \in \Out F_{g,-1} = \operatorname{\normalfont M}^{\text{alg}}_{g,0,0}$
lies in the subgroup
 generated by the image of $\operatorname{\normalfont A}_{g,0} = \Stab(t_1, \Aut F_{g,0})$ together with $\breve \zeta$.
It is clear that
$\phi$  lifts back to an endomorphism $\tilde \phi$ of
$F_{g,0}$ such that $t_1^{\tilde \phi}$ lies in the normal closure of $t_1$.

Now $\operatorname{H}^2(F_{g,-1},\integers) \simeq \integers$;
see, for example,~\cite[Theorem~V.4.9]{DD}.
The image of $\phi$ under the natural map
 $\Aut F_{g,-1} \to \Aut \operatorname{H}^2(F_{g,-1},\integers) \simeq  \{1,-1\}$ is
 denoted $\deg(\phi)$.
By a co\-homology calculation, if we express $t_1^{\tilde \phi}$
as a product of $n_+$ conjugates of $t_1$ and  $n_-$ conjugates of~$\overline t_1$,
then $n_+-n_- = \deg(\phi) = \pm 1$.
By using van Kampen diagrams on a surface,  one can
alter $\tilde \phi$ and arrange that $n_-=0$ or $n_+=0$; this was also done in~\cite[Theorem~4.9]{DG}.
Thus $t_1^{\tilde \phi}$ is now a conjugate of $t_1$ or $\overline t_1$.
 By composing $\tilde \phi$ with an
inner automorphism of $F_{g,0}$, we may assume that $t_1^{\tilde \phi}$ is
$t_1$ or~$\overline t_1$.

Notice that $\zeta  $ lifts back to
$\tilde \zeta \in  \Aut F_{g,0}$ where,
for each   $i \in [1{\uparrow}g]$, $x_i^{\tilde \zeta} \coloneq  y_{g+1-i}$ and
$y_i^{\tilde \zeta} \coloneq   x_{g+1-i} $.  Then
$
\overline t_1^{\tilde \zeta} = (\Pi [x,y]_{\bbl1 1{\uparrow} g \bbr1 })^{\tilde \zeta} =
\Pi [y,x]_{\bbl1 g{\downarrow}1 \bbr1 } = t_1.
$  By replacing $\phi$ with $\phi \zeta$ if necessary, we may now assume that
  $t_1^{\tilde \phi} = t_1$.

We next prove a result, due to Nielsen~\cite{Nielsen0} for $g=1$, and Zieschang~\cite{Z64H} for $g \ge 1$,
that   $t_1^{\tilde \phi} = t_1$ implies that
$\tilde \phi$ is an automorphism of $F_{g,0}$.

We shall show first that $\tilde \phi$ is surjective, by an argument of Formanek~\cite[Theorem~V.4.11]{DD}.
Let $w$ be an element of the  basis $x_{[1 \uparrow g ]} \cup  y_{[1 \uparrow g ]}$ of $F_{g,0}$.
The map of sets \mbox{$ x_{[1 \uparrow g ]} \cup  y_{[1 \uparrow g ]} \to \GL_2(\integers F_{g,0})$,
$v \mapsto
\left( \begin{smallmatrix}
v && 0 \\
\delta_{v,w} && 1
\end{smallmatrix} \right) $} (where $\delta_{v,w}$ equals $1$ if $v = w$ and equals $0$
if   $v \ne w$)
extends uniquely to a group homomorphism
$$ F_{g,0} \to \GL_2(\integers F_{g,0}), \quad v \mapsto
\left( \begin{smallmatrix}
v\phantom{\partial w} && 0 \\
v^{\partial_w} && 1
\end{smallmatrix} \right).$$
The map  $\partial_{w} \colon F_{g,0} \to \integers F_{g,0}$,  called the
\textit{Fox derivative with respect to $w$},
satisfies, for all
 $u$, $v \in  F_{g,0}$,  $(uv)^{\partial_w} = (u^{\partial w})v +  v^{\partial_w} $.
On applying $\partial_w$ to $u \overline u = 1$, we see that $\overline u^{\partial_w} = -u^{\partial_w} \overline u$.
For each $i \in [1{\uparrow}g]$, let $X_i \coloneq x_i^{\tilde \phi}$ and $Y_i \coloneq y_i^{\tilde \phi}$.
Since $\tilde \phi$ fixes $\overline t_1 =  \Pi [x,y]_{\bbl1 1{\uparrow} g \bbr1 }$, we have
$
\Pi [X,Y]_{\bbl1 1{\uparrow} g \bbr1 }
=  \Pi [x,y]_{\bbl1 1{\uparrow} g \bbr1 }.
$
On applying $\partial_w$,   we obtain
\begin{align*}
 & \textstyle  \sum\limits_{i=1}^g  \biggl(\Bigl( X_i^{\partial_w}
\cdot Y_i \cdot (1 - \overline Y_i^{X_iY_i}) +  Y_i^{\partial_w}\cdot (1 - X_i^{Y_i})  \Bigr)
\cdot \Pi [X,Y]_{\bbl1 (i{+}1){\uparrow} g \bbr1 } \biggr)
\\ \nonumber   &\textstyle =   \sum\limits_{i=1}^g  \biggl(\Bigl( x_i^{\partial_w}
\cdot y_i \cdot (1 - \overline y_i^{x_iy_i}) +  y_i^{\partial_w}\cdot (1 - x_i^{y_i})  \Bigr)
\cdot \Pi [x,y]_{\bbl1 (i{+}1){\uparrow} g \bbr1 } \biggr).
\end{align*}
On applying the natural left $\integers F_{g,0}$-linear map $\integers F_{g,0} \to \integers[F_{g,0}/F_{g,0}^{\tilde \phi}]$,
denoted $f{\mapsto}fF_{g,0}^{\tilde \phi}$,  we obtain
\begin{align}
 \label{eq:cons2} & 0
 =        \textstyle\sum\limits_{i=1}^g  \biggl(\Bigl( x_i^{\partial_w}
\cdot y_i \cdot (1 - \overline y_i^{x_iy_i}) +  y_i^{\partial_w}\cdot (1 - x_i^{y_i})  \Bigr)
\cdot \Pi [x,y]_{\bbl1 (i{+}1){\uparrow} g \bbr1 } \biggr)F_{g,0}^{\tilde \phi}.
\end{align}

Consider any $i \in [1{\uparrow}g]$ such that
$x_{[(i{+}1){\uparrow} g]} \cup  y_{[(i{+}1){\uparrow} g]} \subseteq F_{g,0}^{\tilde \phi}$.
By taking $w = y_i$ in~\eqref{eq:cons2}, we obtain
$$0 =
  (1 - x_i^{y_i})
\cdot \Pi [x,y]_{\bbl1 (i{+}1){\uparrow} g \bbr1 } F_{g,0}^{\tilde \phi} = (1 - x_i^{y_i})  F_{g,0}^{\tilde \phi}.$$
Hence $x_i^{y_i} \in F_{g,0}^\phi$, that is, $x_i^{x_iy_i} \in F_{g,0}^\phi$.
By taking $w = x_i$ in~\eqref{eq:cons2} and left multiplying by $\overline y_i$, we obtain
$$0 =
   (1 - \overline y_i^{x_iy_i})
\cdot \Pi [x,y]_{\bbl1 (i{+}1){\uparrow} g \bbr1 } F_{g,0}^{\tilde \phi} = (1 - \overline y_i^{x_iy_i})F_{g,0}^{\tilde \phi}.$$
Hence, $\overline y_i^{x_iy_i} \in  F_{g,p}^\phi$. It follows that $x_i$, $y_i \in F_{g,0}^{\tilde \phi}$.

By induction, $x_{[1{\uparrow} g]} \cup  y_{[1{\uparrow} g]} \subseteq F_{g,0}^{\tilde \phi}$.  Thus
 ${\tilde \phi}$ is
surjective.

 By  Consequence~\ref{cons:z2a}, $\tilde \phi$ is an automorphism, as desired.
\end{proof}

Recall that   $\mathbf{S}_{g,0,p}$ was constructed in Definitions~\ref{defs:bound0} as the
quotient space obtained from   $\mathbf{S}_{g,1,p}$
 by   collapsing  the boundary component to a point.   We then
have  a natural embedding of $\Aut \mathbf{S}_{g,1,p}$ in
$\Aut \mathbf{S}_{g,0,p} $.  Thus   the  Dehn twists and  braid twists  of
$ \mathbf{S}_{g,1,p}$ constructed in Section~\ref{sec:boundary}
induce Dehn twists and  braid twists  of~$\mathbf{S}_{g,0,p}$.
It follows that the image of the ADL set in
$\operatorname{\normalfont M}^{\text{alg}}_{g,0,p}$ lies in $\operatorname{\normalfont M}^{\text{top}}_{g,0,p}$.
Also,   $\breve \zeta$ lies in $\operatorname{\normalfont M}^{\text{top}}_{g,0,p}$, since  $\breve \zeta$
is easily seen to arise from a reflection of $\mathbf{S}_{g,0,p}$.
We now see, in the manner
proposed by  Magnus, Karrass and Solitar~\cite[p.175]{MKS},
 that  the homomorphism $\operatorname{\normalfont M}^{\text{top}}_{g,0,p}
 \to \operatorname{\normalfont M}^{\text{alg}}_{g,0,p}$ is surjective,
by Theorem~\ref{thm:niel}.
 Assuming Epstein's result~\cite{Epstein}, if \mbox{$(g,p) \ne (0,0), (0,1)$},
then $\operatorname{\normalfont M}^{\text{top}}_{g,0,p}$
equals $\operatorname{\normalfont M}^{\text{alg}}_{g,0,p}$, and both are
generated by
the image of the ADLH  set
together with $\breve \zeta$; see~\cite[Corollary~2.11(ii)]{LP}.

\bigskip

 \centerline{\textsc{ {Acknowledgments}}}

\medskip
The research of both authors  was jointly
funded by the MEC (Spain) and the EFRD~(EU) through
Projects  MTM2006-13544 and MTM2008-01550.

We are greatly indebted to  Gilbert Levitt,  Jim McCool, and Luis Paris for
very useful remarks in correspondence and conversations.

\medskip

\noindent \textsc{Llu\'{\i}s Bacardit,
Departament de  Matem\`atiques,
Universitat  Aut\`o\-noma de Bar\-ce\-lo\-na,
E-08193 Bellaterra (Barcelona), Spain
}

\noindent \emph{E-mail address}{:\;\;}\url{lluisbc@mat.uab.cat}

\medskip

\noindent \textsc{Warren Dicks,
Departament de  Matem\`atiques,
Universitat Aut\`o\-noma de Bar\-ce\-lo\-na,
E-08193 Bellaterra (Barcelona), Spain}

\noindent \emph{E-mail address}{:\;\;}\url{dicks@mat.uab.cat}

\noindent \emph{URL}{:\;\;}\url{http://mat.uab.cat/~dicks/}

\end{document}